\documentclass{amsart}
\usepackage{amssymb,stmaryrd,verbatim}

\usepackage{pdflscape}
\usepackage{amsfonts}
\usepackage{amsthm}
\usepackage[inline]{enumitem}
\usepackage[colorlinks,citecolor=blue,urlcolor=black,linkcolor=black]{hyperref}

\usepackage{adjustbox}
\usepackage{graphicx}

\usepackage{xargs}                      
\usepackage[pdftex,dvipsnames]{xcolor}  

\usepackage[colorinlistoftodos,prependcaption,textsize=tiny]{todonotes}
\usepackage{subcaption}

\newcommandx{\unsure}[2][1=]{\todo[linecolor=red,backgroundcolor=red!25,bordercolor=red,#1]{#2}}
\newcommandx{\change}[2][1=]{\todo[linecolor=blue,backgroundcolor=blue!25,bordercolor=blue,#1]{#2}}
\newcommandx{\needtocomplete}[2][1=]{\todo[linecolor=gray,backgroundcolor=gray!25,bordercolor=gray,#1]{#2}}

\usepackage{pdfpages}
\usepackage{tikz-cd}
\usepackage{ifthen}
\newcommand\catalannumber[3]{
	\fill[white!25]  (#1) rectangle +(#2,#2);
	\fill[fill=lightgray]
	(#1)
	\foreach \dir in {#3}{
		\ifnum\dir=0
		-- ++(1,0)
		\else
		-- ++(0,1)
		\fi
	} |- (#1);
	\draw[help lines] (#1) grid +(#2,#2);
	\draw[dashed] (#1) -- +(#2,#2);
	\coordinate (prev) at (#1);
	\foreach \dir in {#3}{
		\ifnum\dir=0
		\coordinate (dep) at (1,0);
		\tikzset{label/.style={below}};
		\else
		\coordinate (dep) at (0,1);
		\tikzset{label/.style={left}};
		\fi
		\draw[line width=2pt,-stealth] (prev) --
		++(dep) coordinate (prev);
	};
}

\newcommand\genallcatalannumbers[6]{%
	\pgfmathtruncatemacro{\steps}{#4+#5}%
	\ifthenelse{#3>\steps}{%
		\ifthenelse{#5<#4}{%
			{%
				\ifthenelse{#4<#2}{%
					\pgfmathtruncatemacro{\nbzero}{#4+1}%
					\genallcatalannumbers{#1}{#2}{#3}{\nbzero}{#5}{#6,0}%
				}{}%
			}%
			{%
				\pgfmathtruncatemacro{\nbone}{#5+1}%
				\genallcatalannumbers{#1}{#2}{#3}{#4}{\nbone}{#6,1}%
			}%
		}{%
			\pgfmathtruncatemacro{\nbzero}{#4+1}%
			\genallcatalannumbers{#1}{#2}{#3}{\nbzero}{#5}{#6,0}%
		}%
	}{%
		\begin{tikzpicture}[myscale]
			\catalannumber{#1}{#2}{#6}
		\end{tikzpicture} %
	}%
}

\newtheorem{theorem}{Theorem}[section]
\newtheorem{claim}{Claim}[theorem]
\newtheorem{subclaim}{Subclaim}[claim]

\newtheorem{lemma}[theorem]{Lemma}
\newtheorem{fact}[theorem]{Fact}
\newtheorem{prop}[theorem]{Proposition}
\newtheorem{cor}[theorem]{Corollary}

\theoremstyle{definition}

\newtheorem{definition}[theorem]{Definition}

\newtheorem*{thma}{Theorem~A}
\newtheorem*{thmb}{Theorem~B}

\theoremstyle{remark}
\newtheorem{remark}[theorem]{Remark}

\newenvironment{subproof}[1][\proofname]{%
	\begin{proof}[#1]%
	}{%
	\end{proof}%
}

\DeclareMathOperator{\reg}{Reg}
\DeclareMathOperator{\card}{Card}
\DeclareMathOperator{\Inner}{in}
\DeclareMathOperator{\non}{non}
\DeclareMathOperator{\supp}{supp}
\DeclareMathOperator{\hu}{hu}
\DeclareMathOperator{\out}{out}
\DeclareMathOperator{\bd}{bd}
\DeclareMathOperator{\add}{add}
\DeclareMathOperator{\otp}{otp}

\DeclareMathOperator{\ch}{CH}

\DeclareMathOperator{\gch}{GCH}
\newcommand*\axiomfont[1]{\textsf{\textup{#1}}}

\newcommand\onto{\axiomfont{onto}}
\newcommand\pfa{\axiomfont{PFA}}

\DeclareMathOperator{\acc}{acc}

\DeclareMathOperator{\zfc}{ZFC}

\DeclareMathOperator{\cf}{cf}

\DeclareMathOperator{\im}{Im}

\DeclareMathOperator{\ns}{NS}

\DeclareCaptionSubType*{figure}

\def\s{\subseteq}

\def\br{\blacktriangleright}

\def\o1{\omega_1}
\begin{document}

\title{Cofinal types below $\aleph_\omega$}
\author{Roy Shalev}
\address{Department of Mathematics, Bar-Ilan University, Ramat-Gan 5290002, Israel.}
\urladdr{https://roy-shalev.github.io/}
\begin{abstract}
It is proved that for every positive integer $n$,
the number of non-Tukey-equivalent directed sets of cardinality $\leq \aleph_n$ 
is at least $c_{n+2}$, the $(n+2)$-Catalan number.
Moreover, the class $\mathcal D_{\aleph_n}$ of directed sets of cardinality $\leq \aleph_n$ 
contains an isomorphic copy of the poset of Dyck $(n+2)$-paths.
Furthermore, we give a complete description whether two successive elements in the copy contain another directed set in between or not.

\end{abstract}
\maketitle

\section{Introduction}
Motivated by problems in general topology,
Birkhoff \cite{MR1503323}, Tukey \cite{MR0002515}, and Day \cite{MR9970}
studied some natural classes of directed sets.
Later, Schmidt \cite{MR76836} and Isbell \cite{MR201316,MR294185} investigated uncountable directed sets under the Tukey order $<_T$.
In \cite{MR792822}, Todor\v{c}evi\'{c} showed that under $\pfa$ there are only five cofinal types in the class $\mathcal D_{\aleph_1}$ of all cofinal types of size $\leq\aleph_1$ under the Tukey order, namely $\{1,\omega,\omega_1,\omega\times \omega_1,[\omega_1]^{<\omega}\}$.
In the other direction, Todor\v{c}evi\'{c} showed that under $\ch$ there are $2^{\mathfrak c}$ many non-equivalent cofinal types in this class.
Later in \cite{MR1407459} this was extended to all transitive relations on $\omega_1$.
Recently, Kuzeljević and Todor\v{c}evi\'{c} \cite{kuzeljevic2021cofinal} initiated the study of the class $\mathcal D_{\aleph_2}$.
They showed in $\zfc$ that this class contains at least fourteen different cofinal types which can be constructed from two basic types of directed sets and their products: $(\kappa,\in )$ and $([\kappa]^{<\theta},\subseteq)$, where $\kappa\in\{1,\omega,\omega_1,\omega_2\}$ and $\theta\in\{\omega,\omega_1\}$.

In this paper, we extend the work of Todor\v{c}evi\'{c} and his collaborators and uncover a connection between the classes of the $\mathcal D_{\aleph_n}$'s and the Catalan numbers. 
Denote $V_k:=\{1, \omega_k, [\omega_k]^{<\omega_m} \mid 0\leq m<k\}$, $\mathcal F_n:=\bigcup_{k\leq n} V_k$ and finally let $\mathcal S_n$ be the set of all finite products of elements of $\mathcal F_n$. 
Recall (see Section~\ref{Section - Catalan}) that the $n$-Catalan number is equal to the cardinality of 
the set of all Dyck $n$-paths.
The set  $\mathcal K_n$ of all Dyck $n$-paths admits a natural ordering $\vartriangleleft$,
and the connection we uncover  is as follows.

\begin{thma} 
	The posets $(\mathcal S_n\slash{\equiv_T}, <_T)$ and $(\mathcal K_{n+2},\vartriangleleft)$ are isomorphic.
	In particular, the class $\mathcal D_{\aleph_n}$ has size at least the $(n+2)$-Catalan number.
\end{thma} 

A natural question which arises is whether an interval determined by two successive elements of $\left(\mathcal S_n\slash{\equiv_T}, {<_T}\right)$
forms an empty interval in $(\mathcal D_{ \aleph_n},<_T)$.
In \cite{kuzeljevic2021cofinal}, the authors showed that there are two intervals of $\mathcal S_2$
that are indeed empty in $\mathcal D_{\aleph_2}$,
and they also showed that consistently, under $\gch$ and the existence of a non-reflecting stationary subset of $E^{\omega_2}_\omega$, 
two intervals of $\mathcal S_2$ that are nonempty in $\mathcal D_{\aleph_2}$.

In this paper, we prove:
\begin{thmb} 
Assuming $\gch$, 	for every positive integer $n$,
all intervals of $\mathcal S_n$
that form an empty interval in $\mathcal D_{\aleph_n}$ are identified,
and counterexamples are constructed to the other cases.
\end{thmb}

\subsection{Organization of this paper} In Section~\ref{Section - Tools} we analyze the Tukey order of directed sets using characteristics of the ideal of bounded subsets.

In Section~\ref{Section - Catalan} we consider the poset $(\mathcal S_n\slash{\equiv_T}, <_T )$ and show it is isomorphic to the poset of good $(n+2)$-paths (Dyck paths) with the natural order. As a Corollary we get that the cardinality of $\mathcal D_{ \aleph_n}$ is greater than or equal to the Catalan number $c_{n+2}$.
Furthermore, we address the basic question of whether a specific interval in the poset $(\mathcal S_n\slash{\equiv_T}, <_T )$ is empty, i.e. considering an element $C$ and a successor of it $E$, is there a directed set $D\in \mathcal D_{\aleph_n}$ such that $C<_T D <_T E$?
We answer this question in Theorem~\ref{Theorem - Intervals} using results from the next two sections.

In Section~\ref{Section - Gaps} we present sufficient conditions on an interval of the poset $(\mathcal S_n\slash{\equiv_T}, <_T )$ which enable us to prove there is no directed set inside.

In Section~\ref{Section - No Gaps} we present cardinal arithmetic assumptions, enough to construct on specific intervals of the poset $(\mathcal S_n\slash{\equiv_T}, <_T )$ a directed set inside.

In Section~\ref{Section - 6} we finish with a remark about future research.

In the Appendix diagrams of the posets $(\mathcal S_2\slash{\equiv_T}, <_T )$ and $(\mathcal S_3\slash{\equiv_T}, <_T )$ are presented.

\subsection{Notation}\label{notationsubsection}
For a set of ordinals $C$, we write $\acc(C):=\{\alpha <\sup(C) \mid\sup(C\cap \alpha)=\alpha >0\}$.
For $\alpha<\gamma$ where $\alpha$ is a regular cardinal, denote $E^\gamma_\alpha:=\{ \beta<\gamma\mid \cf(\beta)=\alpha\}$.
The set of all infinite (resp. infinite and regular) cardinals below $\kappa$ is denoted by $\card(\kappa)$ (resp. $\reg(\kappa)$).
For a cardinal $\kappa$ we denote by $\kappa^+$ the successor cardinal of $\kappa$, and by $\kappa^{+n}$ the $n^{th}$-successor cardinal.
For a function $f:X\rightarrow Y$ and a set $A\subseteq X$, we denote $f"A:=\{f(x)\mid x\in A\}$.
For a set $A$ and a cardinal $\theta$, we write $[A]^{\theta}:=\{X\subseteq A \mid |X|=\theta\}$ and define $[A]^{\leq\theta}$ and $[A]^{<\theta}$ similarly.
For a sequence of sets $\langle A_i\mid i\in A \rangle$, let $\prod_{i\in I} D_i:=\{f:I\rightarrow \bigcup_{i\in I} D_i \mid \forall i\in I[f(i)\in D_i] \}$.

\subsection{Preliminaries}
A partial ordered set $(D,\leq_D)$ is \emph{directed} iff for every $x,y\in D$ there is $z\in D$ such that $x\leq_D z$ and $y\leq_D z$.
We say that a subset $X$ of a directed set $D$ is \emph{bounded} if there is some $d\in D$ such that $x\leq_D d$ for each $x\in X$.
Otherwise, $X$ is \emph{unbounded} in $D$.
We say that a subset $X$ of a directed $D$ is \emph{cofinal} if for every $d\in D$ there exists some $x\in X$ such that $d\leq_D x$.
Let $\cf(D)$ denote the minimal cardinality of a cofinal subset of $D$.
If $D$ and $E$ are two directed sets, we say that $f:D\rightarrow E$ is a \emph{Tukey function} if $f"X:=\{f(x)\mid x\in X\}$ is unbounded in $E$ whenever $X$ is unbounded in $D$.
If such a Tukey function exists we say that $D$ is \emph{Tukey reducible} to $E$, and write $D\leq_T E$.
If $D\leq_T E$ and $E\not\leq_T D$, we write $D<_T E$.
A function $g:E\rightarrow D$ is called a \emph{convergent/cofinal map} from $E$ to $D$ if for every $d\in D$ there is an $e_d\in E$ such that for every $c\geq e_d$ we have $g(c)\geq d$.
There is a convergent map $g:E\rightarrow D$ iff $D\leq_T E$.
Note that for a convergent map $g:E\rightarrow D$ and a cofinal subset $Y\subseteq E$, the set $g"Y$ is cofinal in $D$.
We say that two directed sets $D$ and $E$ are \emph{cofinally/Tukey equivalent} and write $D\equiv_T E$ iff $D\leq_T E$ and $D\geq_T E$.
Formally, a \emph{cofinal type} is an equivalence class under the Tukey order, we abuse the notation and call every representative of the class a cofinal type.
Notice that a directed set $D$ is cofinally equivalent to any cofinal subset of $D$.
In \cite{MR0002515}, Tukey proved that $D\equiv_T E$ iff there is a directed set $(X, \leq_X)$ such that both $D$ and $E$ are isomorphic to a cofinal subset of $X$.
We denote by $\mathcal D_\kappa$ the set of all cofinal types of directed sets of cofinality $\leq \kappa$.

Consider a sequence of directed sets $\langle D_i \mid i\in I \rangle$, we define the directed set which is the product of them $(\prod_{i\in I} D_i,\leq)$ ordered by everywhere-dominance, i.e. for two elements $d, e \in\prod_{i\in I} D_i$ we let $ d \leq  e$ if and only if $ d(i)\leq_{D_i} e(i)$ for each $i\in I$.
For $X\subseteq \prod_{i\in I} D_i$, let $\pi_{D_i}$ be the projection to the $i$-coordinate.
A simple observation {\cite[Proposition~2]{MR792822}} is that if $n$ is finite, then $D_1\times \dots \times D_n$ is the least upper bound of $D_1,\dots,D_n$ in the Tukey order.
Similarly, we define a $\theta$-support product $\prod^{\leq\theta}_{i\in I}D_i$; for each $i\in I$, we fix some element $0_{D_i}\in D_i$ (usually minimal).
Every element $v\in \prod^{\leq\theta}_{i\in I}D_i$ is such that $|\supp(v)|\leq \theta$, where $\supp(v):=\{ i\in I \mid v(i)\not = 0_{D_i}\}$.
The order is coordinate wise.

\section{Characteristics of directed sets}\label{Section - Tools}

We commence this section with the following two Lemmas which will be used throughout the paper.
 
\begin{lemma}[Pouzet, {\cite{Pouzet}}]\label{Lemma - cofinal set with no large bounded set}
	Suppose $D$ is a directed set such that $\cf(D)=\kappa$ is infinite, then there exists a cofinal directed set $P\subseteq D$ of size $\kappa$ such that every subset of size $\kappa$ of $P$ is unbounded
\end{lemma}
\begin{proof}
	Let $X\subseteq D$ be a cofinal subset of cardinality $\kappa$ and let $\{x_\alpha \mid \alpha<\kappa\}$ be an enumeration of $X$.
	Let $P:=\{ x_\alpha \mid \alpha<\kappa\text{ and for all } \beta<\alpha[x_\alpha\not <_D x_\beta] \}$.
	We claim that $P$ is cofinal.
	In order to prove this, fix $d\in D$.
	As $X$ is cofinal in $D$, fix a minimal $\alpha<\kappa$ such that $d<_D x_\alpha$.
	If $x_\alpha\in P$, then we are done.
	If not, then fix some $\beta<\alpha$ minimal such that $x_\alpha<_D x_\beta$.
	We claim that $x_\beta\in P$, i.e. there is no $\gamma<\beta$ such that $x_\beta<_D x_\gamma$. 
	Suppose there is some $\gamma<\beta$ such that $x_\beta<_D x_\gamma$, then $x_\alpha <_D x_\gamma$, which is a contradiction to the minimality of $\beta$.
	Note that $d<_D x_\beta\in P$ as sought.
	As $P$ is cofinal in $D$, $\cf(D)=\kappa$, $P\subseteq X$ and $|X|=\kappa$, we get that $|P|=\kappa$.
	
	Finally, let us show that every subset of size $\kappa$ of $P$ is unbounded.
	Suppose on the contrary that $X\subseteq P$ is a bounded subset of $P$ of size $\kappa$.
	Fix some $x_\beta\in P$ above $X$ and $\beta<\alpha<\kappa$ such that $x_\alpha \in X$, but this is an absurd as $x_\alpha<_D x_\beta$ and $x_\alpha\in P$.
\end{proof}

\begin{fact}[Kuzeljević-Todor\v{c}evi\'{c}, {\cite[Lemma~2.3]{kuzeljevic2021cofinal}}]\label{KT2021 - Lemma 2.3}
	Let $\lambda\geq \omega$ be a regular cardinal and $n<\omega$ be positive.
	The directed set $[\lambda^{+n}]^{\leq \lambda}$ contains a cofinal subset $\mathfrak D_{[\lambda^{+n}]^{\leq \lambda}}$ of size $\lambda^{+n}$ with the property that every subset of $\mathfrak D_{[\lambda^{+n}]^{\leq \lambda}}$ of size $>\lambda$ is unbounded in $[\lambda^{+n}]^{\leq \lambda}$.
	In particular, $[\lambda^{+n}]^{\leq \lambda}$ belongs to $\mathcal D_{\lambda^{+n}}$, i.e. $\cf([\lambda^{+n}]^{\leq \lambda}) \leq \lambda^{+n}$.
\end{fact}
Recall that any directed set is Tukey equivalent to any of its cofinal subsets, hence $\mathfrak D_{[\lambda^{+n}]^{\leq \lambda}} \equiv_T [\lambda^{+n}]^{\leq \lambda}$.

As part of our analysis of the class $\mathcal D_{\aleph_n}$, we would like to find certain traits of directed sets which distinguish them from one another in the Tukey order.
This was done previously, in \cite{MR76836}, \cite{MR201316} and \cite{MR1407459}. 
We use that the language of cardinal functions of ideals.

\begin{definition}For a set $D$ and an ideal $\mathcal I$ over $D$, consider the following cardinal characteristics of $\mathcal I$:
	\begin{itemize}	
		\item $\add(\mathcal I):=\min\{\kappa \mid \mathcal A \subseteq I,~ |\mathcal A|=\kappa, ~\bigcup \mathcal A\notin \mathcal I\}$;
		\item $ \non(\mathcal I):=\min \{|X|\mid  X\subseteq D,~ X\notin \mathcal I  \}$;
		\item $  \out(\mathcal I):=\min\{ \theta\leq|D|^+\mid \mathcal I\cap[D]^{\theta}=\emptyset\}$;
		\item $\Inner(\mathcal I, \kappa) = \{ \theta\leq \kappa \mid \forall X\in [D]^\kappa~\exists Y\in [X]^\theta \cap \mathcal I \}$.

	\end{itemize}
\end{definition}

Notice that $\add(\mathcal I)\leq \non(\mathcal I) \leq \out(\mathcal I)$.

\begin{definition}
	For a directed set $D$, denote by $\mathcal I_{\bd}(D) $ the ideal of bounded subsets of $D$.
\end{definition}

\begin{prop}
	Let $D$ be a directed set. Then:
	\begin{enumerate}
		\item  $\non(\mathcal I_{\bd}(D))$ is the minimal size of an unbounded subset of $D$, so every subset of size less than $\non(\mathcal I_{\bd}(D))$ is bounded;
		\item If $\theta<\out(\mathcal I_{\bd}(D))$, then there exists in $D$ some bounded subset of size $\theta$. 
		\item If $\theta \geq \out(\mathcal I_{\bd}(D))$, then every subset $X$ of size $\theta$ is unbounded in $D$;
		\item If $\theta\in \Inner(\mathcal I_{\bd}(D),\kappa)$, then for every $X\in [D]^\kappa$ there exists some $B\in [X]^{\theta}$ bounded;
		\item For every $\theta<\add(\mathcal I_{\bd}(D))$ and a family $\mathcal A$ of size $\theta$ of bounded subsets of $D$, the subset $\bigcup \mathcal A$ is also bounded in $D$. \qed
	\end{enumerate}
\end{prop}

Let us consider another intuitive feature of a directed set, containing information about the cardinality of hereditary unbounded subsets, this was considered previously by Isbell \cite{MR201316}.
\begin{definition}[Hereditary unbounded sets]
	For a directed set $D$, set $$ \hu(D):=\{ \kappa\in \card(|D|^+)\mid \exists X\in [D]^{\kappa}[\forall Y\in [X]^\kappa\text{ is unbounded}] \}.$$
\end{definition}

\begin{prop}
	Let $D$ be a directed set. Then:
	\begin{itemize}	
		\item If $\cf(D)$ is an infinite cardinal, then $\cf(D)\in  \hu(D)$;
		\item If $\out(\mathcal I_{\bd}(D))\leq\kappa \leq |D|$, then $ \kappa \in \hu(D)$;
		\item For an infinite cardinal $\kappa$ we have that $\non(\mathcal I_{\bd}(\kappa))=\cf(\kappa)$, $\out(\mathcal I_{\bd}(D))=\kappa$ and
		$\hu(\kappa)=\{\lambda\in \card(\kappa^+) \mid \lambda =\cf(\kappa)\}$;
		\item If $\kappa=\cf(D)=\non(\mathcal I_{\bd}(D))$, then $D\equiv_T\kappa$;
		\item For two infinite cardinals $\kappa>\theta$ we have that $\non(\mathcal I_{\bd}([\kappa]^{<\theta}))=\cf(\theta)$;
		\item For a regular cardinal $\kappa$ and a positive $n<\omega$, $\out(\mathcal I_{\bd}(\mathfrak D_{[\kappa^{+n}]^{\leq \kappa}}))>\kappa$ and $\hu(\mathfrak D_{[\kappa^{+n}]^{\leq \kappa}}) = \{ \kappa^{+(m+1)}\mid m< n\}$;
		\item If $\kappa=\cf(D)$ is regular, $\theta=\out(\mathcal I_{\bd}(D))=\non(\mathcal I_{\bd}(D))$ and $\theta^{+n}=\kappa$ for some $n<\omega$, then $D\equiv_T [\kappa]^{<\theta}$.\qed
	\end{itemize}
\end{prop}

In the rest of this section we consider various scenarios in which the traits of a certain directed set give us information about its position in the poset $(\mathcal D_{\kappa}, <_T )$.
 
\begin{lemma}\label{Lemma - existence of full-unbounded kappa set}
	Suppose $D$ is a directed set, $\kappa$ is an infinite regular cardinal and $X\subseteq D$ is an unbounded subset of size $\kappa$ such that every subset of $X$ of size $<\kappa$ is bounded.
	Then $\kappa\in \hu(D)$.
\end{lemma}
\begin{proof}
	Enumerate $X:=\{x_\alpha \mid \alpha<\kappa\}$, by the assumption, for every $\alpha<\kappa$ we may fix some $z_\alpha\in D$ above the bounded initial segment $\{x_\beta\mid \beta<\alpha\}$.
	We show that $Z:=\{z_\alpha\mid \alpha<\kappa\}$, witnesses $\kappa\in \hu(D)$.
	First, let us show that $|Z|=\kappa$. 
	Suppose on the contrary that $Z:=\{z_\alpha\mid \alpha<\kappa\}$ is of cardinality $< \kappa$.
	Then for some $\alpha<\kappa$, the element $z_\alpha$ is above the subset $X$, hence $X$ is bounded which is absurd.
	Now let us prove that $Z$ is hereditarily unbounded.
	We claim that every subset of $Z$ of cardinality $\kappa$ is also unbounded.
	Suppose not, let us fix some $W\in [Z]^\kappa$ bounded by some $d\in D$, but then $d$ is above $X$ contradicting the fact that $X$ is unbounded.
\end{proof}

\begin{lemma}\label{Lemma - kappa leq_T D}
	Suppose $D$ is a directed set and $\kappa$ is an infinite cardinal in $\hu(D)$, then $\kappa \leq_TD$.
\end{lemma}
\begin{proof}
	Fix $X\subseteq D$ of cardinality $\kappa$ such that every subset of $X$ of size $\kappa$ is unbounded and a one-to-one function $f:\kappa\rightarrow X$, notice that $f$ is a Tukey function from $\kappa$ to $D$ as sought.
\end{proof}

\begin{cor}\label{Cor - existence of full-unbounded kappa set}
	Suppose $D$ is directed set, $\kappa$ is regular and $X\subseteq D$ is an unbounded subset of size $\kappa$ such that every subset of $X$ of size $< \kappa$ is bounded, then $\kappa\leq_T D$.\qed
\end{cor}

The reader may check the following:
\begin{itemize} 
	\item For any two infinite cardinals $\lambda$ and $\kappa$ of the same cofinality, we have $\lambda \equiv_T \kappa$;
	\item For an infinite regular cardinal $\kappa$, we have $\kappa\equiv_T [\kappa]^{<\kappa}$.
	
	\item $\hu(\prod^{<\omega}_{n<\omega}\omega_{n+1}) = \{\omega_n \mid n<\omega\}$.
\end{itemize}

\begin{lemma}
	Suppose $D$ and $E$ are two directed sets such that for some $\theta\in \hu(D)$ regular we have $\theta>\cf(E)$, then $D\not\leq_T E$.
\end{lemma}
\begin{proof} By passing to a cofinal subset, we may assume that $|E|=\cf(E)$.
	Fix $\theta\in \hu(D)$ regular such that $\cf(E)<\theta$ and $X\in [D]^\theta$ witnessing $\theta\in \hu(D)$, i.e. every subset of $X$ of size $\theta$ is unbounded.
	Suppose on the contrary that there exists a Tukey function $f:D\rightarrow E$.
	By the pigeonhole principle, there exists some $Z\in [X]^{\theta}$ and $e\in E$ such that $f"Z=\{e\}$.
	As $f$ is Tukey and the subset $Z\subseteq X$ is unbounded, $f"Z$ is unbounded in $E$ which is absurd.	
\end{proof}

Notice that for every directed set $D$, if $\cf(D)>1$, then $\cf(D)$ is an infinite cardinal.

As a corollary from the previous Lemma, $\lambda \not \leq_T \kappa$ for any two regular cardinals $\lambda>\kappa$ where $\lambda$ is infinite.
Furthermore, the reader can check that $\lambda\not\leq_T\kappa$, whenever $\lambda<\kappa$ are infinite regular cardinals.
\begin{lemma}\label{Lemma - C leq_T D imply cf(C) leq cf(D)}
	Suppose $C$ and $D$ are directed sets such that $C\leq_T D$, then $\cf(C)\leq \cf(D)$.
\end{lemma}
\begin{proof}
	Suppose $|D|=\cf(D)$ and let $f:C\rightarrow D$ be a Tukey function.
	As $f$ is Tukey, for every $d\in D$ the set $\{x\in C\mid f(x)=d\}$ is bounded in $C$ by some $c_d\in C$.
	Note that for every $x\in C$, we have $x\leq_C c_{f(x)}$, hence the set $\{c_d \mid d\in D\}$ is cofinal in $C$.
	So $\cf(C)\leq |D|=\cf(D)$ as sought.
\end{proof}

\begin{lemma}\label{Lemma - below kappa theta}
	Let $\kappa$ and $\theta$ be two cardinals such that $\theta<\kappa=\cf(\kappa)$.
	
	Suppose $D$ is a directed set such that $\cf(D)\leq\kappa$ and $\non(\mathcal I_{\bd}(D))\geq\theta$,

	then $D\leq_T [\kappa]^{<\theta}$.
	Furthermore, if $\theta\in \Inner(\mathcal I_{\bd}(D),\kappa)$, then $D<_T [\kappa]^{<\theta}$.
\end{lemma}
\begin{proof}
	First we show that there exists a Tukey function $f:D \rightarrow [\kappa]^{<\theta}$.
	Let us fix a cofinal subset $X\subseteq D$ of cardinality $\leq\kappa$ such that every subset of $X$ of cardinality $< \theta$ is bounded.
	As $|X|\leq \kappa$ we may fix an injection $f:X\rightarrow [\kappa]^{1}$, we will show $f$ is a Tukey function.
	Let $Y\subseteq X$ be a subset unbounded in $D$, this implies $|Y|\geq\theta$.
	As $f$ is an injection, the set $\bigcup f"Y$ is of cardinality $\geq \theta$.
	Note that every subset of $[\kappa]^{<\theta}$ whose union is of cardinality $\geq\theta$ is unbounded in $[\kappa]^{<\theta}$, hence $f"Y$ is an unbounded subset in $[\kappa]^{<\theta}$ as sought.
	
	Assume $\theta\in \Inner(\mathcal I_{\bd}(D),\kappa)$, we are left to show that $[\kappa]^{<\theta}\not\leq_T D$.
	Suppose on the contrary that $g:[\kappa]^{<\theta}\rightarrow D$ is a Tukey function.
	We split to two cases:
	
	$\br$ Suppose $| g" [\kappa]^1| < \kappa$. As $\kappa$ is regular, by the pigeonhole principle there exists a set $X\subseteq [\kappa]^1$ of cardinality $\kappa$,  and $d\in D$ such that $g(x)=d$ for each $x\in X$. Notice $g"X$ is a bounded subset of $D$.
	As $X\subseteq [\kappa]^1$ is of cardinality $\kappa$ and $\kappa>\theta$, it is unbounded in $[\kappa]^{<\theta}$.
	Since $g$ is a Tukey function, we get that $g"X$ is unbounded which is absurd.
	
	$\br$ Suppose $|g"[\kappa]^1| = \kappa$. Let $X:=g"[\kappa]^1$, by our assumption on $D$, there exists a bounded subset $B\in[X]^{\theta}$.
	Since $B$ is of size $\theta$, we get that $(g^{-1} [B])\cap [\kappa]^1$ is of cardinality $\geq\theta$, hence unbounded in $[\kappa]^{<\theta}$, which is absurd to the assumption $g$ is Tukey.
\end{proof}

\begin{remark}
	For every two directed sets, $D$ and $E$, if $\non(\mathcal I_{\bd}(D) )<\non(\mathcal I_{\bd}(E) )$, then $D\not \leq_T E$.
	For example, $\theta\not\leq_T [\kappa]^{\leq\theta}$.
\end{remark}

\begin{lemma}\label{Lemma - kappa, theta no tukey map}
	Let $\kappa$ be a regular infinite cardinal.
	Suppose $D$ and $E$ are two directed sets such that $|D|\geq\kappa$ and $\out(\mathcal I_{\bd}(D))\in \Inner(\mathcal I_{\bd}(E),\kappa)$, then $D\not\leq_T E$.
\end{lemma}
\begin{proof}
	Let $\theta:=\out(\mathcal I_{\bd}(D))$.
	By the definition of $\Inner(\mathcal I_{\bd}(E),\kappa)$, as $\theta\in \Inner(\mathcal I_{\bd}(E),\kappa)$, we know that $\theta\leq \kappa$.
	Notice that every subset of $D$ of size $\geq\theta$ is unbounded in $D$ and every subset of size $\kappa$ of $E$ contains a bounded subset in $E$ of size $\theta$.
	
	Suppose on the contrary that there exists a Tukey function $f:D\rightarrow E$.
	We split to two cases:
	
	$\br$ Suppose $|f"D|<\kappa$, then by the pigeonhole principle there exists some $X\in [D]^\kappa$ and $e\in E$ such that $f"X=\{e\}$.
	As $|X|=\kappa\geq \theta$, we know that $X$ is unbounded in $D$, but $f"X$ is bounded in $E$ which is absurd as $f$ is a Tukey function.
	
	$\br$ Suppose $|f"D|\geq\kappa$, by the assumption there exists a subset $Y\in [f"D]^\theta$ which is bounded in $E$.
	Notice that $X:=f^{-1}Y$ is a subset of $D$ of size $\geq \theta$, hence unbounded in $D$.
	So $X$ is an unbounded subset of $D$ such that $f"X=Y$ is bounded in $E$, contradicting the fact that $f$ is a Tukey function.
\end{proof}
\begin{lemma}\label{Lemma - out(, kappa)}
	Suppose $\kappa$ is a regular uncountable cardinal, $C$ and $\langle D_m \mid m<n\rangle$ are directed sets such that $|C|<\kappa\leq \cf(D_m)$ and $\non(\mathcal I_{\bd}(D_m))>\theta$ for every $m<n$.
	Then $\theta\in \Inner(\mathcal I_{\bd} (C\times \prod_{m<n}{D_m}),\kappa) $.
\end{lemma}
\begin{proof}
	Suppose $X\subseteq C\times \prod_{m<n}{D_m}$ is of size $\kappa$, we show that $X$ contains a bounded subset of size $\theta$.
	As $|C|<\kappa$, by the pigeonhole principle we can fix some $Y\in [X]^\kappa$ and $c\in C$ such that $\pi_C " Y=\{c\}$.
	Suppose on the contrary that some subset $Z\subseteq Y$ of size $\theta$ is unbounded, it must be that for some $m<n$ the set $\pi_{D_m}"Z$ is unbounded in $D_m$, but this is absurd as $\non(\mathcal I_{\bd}(D_m)) >\theta$ and $|\pi_{D_m}"Z|\leq \theta$.
\end{proof}

\begin{lemma}\label{Lemma - D not<= C x E}
	Suppose $C,D$ and $E$ are directed sets such that:
	\begin{itemize}
		\item for every partition $D=\bigcup_{\gamma<\kappa} D_\gamma$, there exists an ordinal $\gamma<\kappa$, and an unbounded $X\subseteq D_\gamma$ of size $\kappa$;
		\item $|C|\leq \kappa$;
		\item $\non(\mathcal I_{\bd}(E))> \kappa$
	\end{itemize}
	Then $D\not \leq_T C\times E$.
\end{lemma}
\begin{proof}
	Suppose on the contrary, that there exists a Tukey function $h:D \rightarrow C \times E$.
	For $c\in C$, let $D_c:=\{x\in D \mid \exists e\in E [h(x)=(c,e)]\}$.
	Since $h$ is a function, $D:=\bigcup_{c\in C} D_c$ is a partition to $\leq\kappa$ many sets.
	By the assumption, there exists $c\in C$ and an unbounded subset $X\subseteq D_c$ of cardinality $\kappa$.
	Enumerate $X=\{x_\xi \mid \xi<\kappa\}$ and let $e_\xi\in E$ be such that $h(x_\xi) =(c,e_\xi)$, for each $\xi<\kappa$. 
	As $\non(\mathcal I_{\bd}(E))> \kappa$, there exists some upper bound $e\in E$ to the set $\{e_\xi \mid \xi<\kappa\}$.
	Since $X$ is unbounded and $h$ is Tukey, $h"X=\{(c,e_\xi )\mid \xi<\kappa\}$ must be unbounded, which is absurd as $(c,e)$ is bounding it.
\end{proof}

Note that the Lemma is also true when the partition of $D$ is of size less than $\kappa$.

\section{The Catalan structure}\label{Section - Catalan}

The sequence of Catalan numbers $\langle c_n \mid n<\omega \rangle=\langle 1,1,2,5,14,42,\dots\rangle$ is an ubiquitous sequence of integers with many characterizations, for a comprehensive review of the subject, we refer the reader to Stanley's book \cite{MR3467982}.
One of the many representations of $c_n$, is the number of good $n$-paths (Dyck paths), where a \emph{good $n$-path} is a monotonic lattice path along the edges of a grid with $n\times n$ square cells, which do not pass above the diagonal.
A \emph{monotonic path} is one which starts in the lower left corner, finishes in the upper right corner, and consists entirely of edges pointing rightwards or upwards.
An equivalent representation of a good $n$-path, which we will consider from now on, is a vector $\vec p$ of the columns' heights of the path (ignoring the first trivial column), i.e. a vector $\vec p=\langle p_0,\dots ,p_{n-2}\rangle$ of length $n-1$ of $\leq$-increasing numbers satisfying $0\leq p_k\leq k+1$, for every $0\leq k\leq n-2$.
We consider the poset $(\mathcal K_n,\vartriangleleft)$ where $\mathcal K_n$ is the set of all good $n$-paths and the relation $\vartriangleleft$ is defined such that $\vec a \vartriangleleft \vec b$ if and only if the two paths are distinct and for every $k$ with $0\leq k\leq n-2$ we have $b_k\leq a_k$, in other words, the path $\vec b$ is below the path $\vec a$ (allowing overlaps).
Notice that for two distinct good $n$-paths $\vec a$ and $\vec b$, either $\vec a \not \vartriangleleft \vec b$ or $\vec b \not \vartriangleleft \vec a$.
A good $n$-path $\vec b$ is an immediate successor of a good $n$-path $\vec a$ if $\vec a\vartriangleleft \vec b$ and $\vec a-\vec b$ is a vector with value $0$ at all coordinates except one of them which gets the value $1$.

Suppose $\vec a$ and $\vec b$ are two good $n$-paths where $\vec b$ is an immediate successor of $\vec a$.
Let $i\leq n-2$ be the unique coordinate on which $\vec a$ and $\vec b$ are different and $a_i$ be the value of $\vec a$ on this coordinate, i.e. $a_i=b_i+1 $.
We say that the pair $(\vec a, \vec b)$ is on the $k$-diagonal if and only if $i+1-a_i=k$ and $\vec b$ is an immediate successor of $\vec a$.
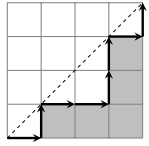
\begin{figure}[h]
		\adjustbox{scale=0.45, center}{\begin{tikzcd}
				\catalannumber{	0,0}{4}{0,1,0,0,1,1,0,1}
		\end{tikzcd}}
		\caption{The good $4$-path $\langle 1,1,3\rangle$.}
\end{figure}

In this section we show the connection between the Catalan numbers and cofinal types.
Let us fix $n<\omega$.
Recall that for every $k<\omega$, we set $V_k:=\{1, \omega_k, [\omega_k]^{<\omega_m} \mid 0\leq m<k\}$, $\mathcal F_n:=\bigcup_{k\leq n} V_k$ and let $\mathcal S_n$ be the set of all finite products of elements in $\mathcal F_n$.
Our goal is to construct a coding which gives rise to an order-isomorphism between $(\mathcal S_n\slash{\equiv_T}, <_T)$ and $(\mathcal K_{n+2},\vartriangleleft)$.

To do that, we first consider a ``canonical form" of directed sets in $\mathcal S_n$.
By Lemma~\ref{Lemma - below kappa theta} the following hold:
\begin{enumerate}[label=(\alph*)]
	\item\label{Clause - Catalan 1} for all $0\leq l<m<k<\omega$ we have $1<_T \omega_k<_T [\omega_k]^{<\omega_m} <_T [\omega_k]^{<\omega_l} $.
	\item\label{Clause - Catalan 2} for all $0\leq l\leq t<m\leq k<\omega$ with $(l,k)\neq(t,m)$ we have $[\omega_m]^{<\omega_t} <_T [\omega_k]^{<\omega_l}$ and $\omega_m <_T [\omega_k]^{<\omega_l}$.
\end{enumerate}

Notice that \ref{Clause - Catalan 1} implies $(V_k,<_T)$ is linearly ordered.
A basic fact is that for two directed sets $C$ and $D$ such that $C\leq_T D$, we have $C\times D \equiv_T D$.
Hence, for every $D\in \mathcal S_n$ we can find a sequence of elements $\langle D^k \mid k\leq n \rangle$, where $D^k\in V_k$ for every $k\leq n$, such that $D\equiv_T \prod_{k\leq n}D^k$.
As we are analyzing the class $\mathcal D_{\aleph_n}$ under the Tukey relation $<_T$, two directed sets which are of the same $\equiv_T$-equivalence class are indistinguishable, so from now on we consider only elements of this form in $\mathcal S_n$.

We define a function $\mathfrak F:\mathcal S_n \rightarrow \mathcal S_n$ as follows:
Fix $D\in \mathcal S_n$ where $D=\prod_{k\leq n} D^k$.
Next we construct a sequence $\langle D_k \mid k\leq n \rangle$ by reverse recursion on $k\leq n$.
At the top case, set $D_n:=D^n$.
Next, for $0\leq k< n$. If by \ref{Clause - Catalan 2}, we get that $D^{k} <_T D^{m}$ for some $k<m\leq n$, then set $D_{k}:=1$.
Else, let $D_{k}:=D^{k}$.
Finally, let $\mathfrak F(D) := \prod_{k\leq n} D_{k}$.
Notice that we constructed $\mathfrak F(D)$ such that $\mathfrak F(D) \equiv_T D$.
We define $\mathcal T_n:=\im(\mathfrak F)$.

\medskip\noindent\textbf{The coding.} 
	We encode each product $D\in \mathcal T_n$ by an $(n+2)$-good path $\vec v_D:=\langle v_0,\dots, v_{n}\rangle$.
	Recall that $D:=\prod_{k\leq n} D_k$, where $D_k\in V_k$ for every $k\leq n$.
	We define by reverse recursion on $0\leq k \leq n$, the elements of the vector $\vec v_D$ such that $v_k\leq k+1$ as follows:
	Suppose one of the elements of $\langle [\omega_k]^{<\omega},\dots, [\omega_k]^{<\omega_{k-1}} ,\omega_k \rangle$ is equal to $D_k$, then let $v_{k}$ be its coordinate (starting from $0$).
	Suppose this is not the case, then if $k=n$, we let $v_{k}:=n+1$ else $v_{k}:=\min \{ v_{k+1}, k+1\}$.

	Notice that by \ref{Clause - Catalan 2}, if $0\leq i<j\leq n$, then $v_i \leq v_j$.
	Hence, every element $D\in \mathcal T_n$ is encoded as a good $(n+2)$-path.

	To see that the coding is one-to-one, suppose $C,D\in \mathcal T_n$ are distinct.
	Let $k:=\max \{ i\leq n \mid C_i\neq D_i\}$.
	We split to two cases:
	
	$\br$ Suppose both $C_k$ and $D_k$ are not equal to $1$, then clearly the column height of $\vec v_C$ and $\vec v_D$ are different at coordinate $k+1$.
	
	$\br$ Suppose one of them is equal to $1$, say $C_k$, then $D_k\neq 1$.
	Let $\vec v_C:=\langle v^C_0,\dots v^C_n\rangle$ and $\vec v_D:=\langle v^D_0,\dots v^D_n\rangle$.
	Suppose $k=n$, then clearly $v^D_n < v^C_n$.
	Suppose $k<n$, then $v^D_i = v^C_i$ for $k<i\leq n$.
	By the coding, $v^D_k<k+1$ and by \ref{Clause - Catalan 2} $v^D_k<v^D_{k+1}=v^C_{k+1}$, but $v^C_k:=\min\{ k+1, v^C_{k+1}\}$.
	Hence $v^D_k < v^C_k$ as sought.
	
	To see that the coding is onto, let us fix a good $(n+2)$-path $\vec v:=\langle v_0, \dots, v_n \rangle$.
	We construct $\langle D_k \mid k\leq n \rangle$ by reverse recursion on $k\leq n$.
	At the top case, set $D_n$ to be the $v_n$ element of the vector $\langle [\omega_n]^{<\omega},\dots, [\omega_n]^{<\omega_{n-1}} ,\omega_n, 1\rangle$.
	For $k<n$, if $v_k=v_{k+1}$, let $D_{k}:=1$.
	Else, let $D_k$ be the $k$th element of the vector $\langle [\omega_k]^{<\omega},\dots, [\omega_k]^{<\omega_{k-1}} ,\omega_k, 1\rangle$.
	Let $D=\prod_{k\leq n} D_k$, notice that as $\vec v$ represents a good $(n+2)$-path we have $D=\mathfrak F(D)$, hence $D \in \mathcal T_n$.
	Furthermore, $\vec v_D =\vec v$, hence the coding is onto as sought.
	As a Corollary we get that $|\mathcal T_n|=c_{n+2}$.

In Figure~2  we present all good $4$-paths and the corresponding types in $\mathcal T_2$ they encode.

\begin{lemma}\label{Lemma - path < implies Tukey <}
	Suppose $C,D\in \mathcal T_n$ and $\vec v_D \vartriangleleft \vec v_C$, then $D\leq_T C$.
\end{lemma}
\begin{proof}
	Let $D = \prod_{k\leq n} D_k$ and $C=\prod_{k\leq n} C_k$.
	Note that if $D_k\leq_T C$ for every $k\leq n$, then $D\leq_T C$ as sought.
	Fix $k\leq n$, if $D_k=1$, then clearly $D_k\leq C$.
	Suppose $D_k\neq 1$, we split to two cases:

	$\br$ Suppose $C_k\neq 1$. As $v^C_k <v^D_k$ and by \ref{Clause - Catalan 1} we have $D_k \leq_T C_k\leq_T C$ as sought.
	
	$\br$ Suppose $C_k=1$, let $m:=\max\{i\leq n \mid k< i,~v^C_i= v^C_k\}$.
	As $v^C_i \leq i+1$, by the coding $m$ is well-defined and $v^C_m= v^C_k\leq k<m$.
	Notice that $C_m=[\omega_m]^{<\omega_p}$ where $p=v^C_m$ and $D_k\equiv_T [\omega_k]^{<\omega_p}$.
	So by~\ref{Clause - Catalan 2}, $D_k\leq_T C_m\leq_T  C$ as sought.
\end{proof}

\begin{lemma}\label{Lemma - path not < implies Tukey not <}
	Suppose $C,D\in \mathcal T_n$ and $\vec v_D \not \vartriangleleft \vec v_C$, then $D\not \leq_T C$.
\end{lemma}
\begin{proof}
	Let $D=\prod_{k\leq n} D_k$, $C=\prod_{k\leq n} C_k$, $\vec v_C:=\langle v_0^C,\dots, v_n^C\rangle$ and $\vec v_D:=\langle v_0^D,\dots, v_n^D\rangle$
	As $\vec v_D \not \vartriangleleft \vec v_C$, we can define $i=\min\{k\leq n \mid v^C_k>v^D_k\}$.
	
	Let $p:=v^D_i$ and $r=\max\{k\leq n \mid v_i^D = v_k^D\}$, notice that $p\leq i$.
	We define a directed set $F$ such that $F\leq_T D$.

	$\br$ Suppose $p=i$ and let $F=\omega_i$.
	If $r=i$, then clearly $F=D_i$ and $F\leq_T D$ as sought.
	Else, by the coding $D_r=[\omega_r]^{<\omega_p}$.
	By Lemma~\ref{Lemma - below kappa theta}, we have $F\leq_T D$ as sought.
	
	$\br$ Suppose $p<i$ and let $F=\mathfrak D_{[\omega_i]^{<\omega_{p}}}$.
	By the coding $D_r=[\omega_r]^{<\omega_p}$ and by Clause~\ref{Clause - Catalan 2}, we have $F\leq_T D$ as sought.
	
	Notice that $\out(\mathcal I_{\bd}(F)) = \omega_p$ and $\cf(F)=\omega_i$.
	As $F\leq_T D$, it is enough to verify that $F\not\leq_T C$.
	
	As $\vec v_C$ is a good $(n+2)$-path, we know that $v_k^C>p$ for every $k\geq i$.
	Consider $A:=\{i\leq k\leq n \mid C_k \neq 1  \}$.
	We split to two cases:
	
	$\br$ Suppose $A=\emptyset$. 
	Then $\cf(\prod_{k\leq n} C_k) <\omega_i$.
	As $\cf(F)=\omega_i$, by Lemma~\ref{Lemma - C leq_T D imply cf(C) leq cf(D)} we have that $F\not \leq_T \prod_{k\leq n} C_k$ as sought.
	
	$\br$ Suppose $A\neq \emptyset$.
	Let $E:=\prod_{k<i} C_k \times \prod_{k\in A} C_k$
	Notice that $\cf(\prod_{k<i} C_k)<\omega_i$, $\prod_{i\leq k \leq n} C_k \equiv_T \prod_{k\in A} C_k$ and $C\equiv_T E$.
	Furthermore, for each $k\in A$, we have $\non(\mathcal I_{bd}(C_k))>\omega_p$.
	By Lemma~\ref{Lemma - out(, kappa)}, we have $\omega_p\in \Inner(\mathcal I_{\bd}(E),\omega_i)$.
	Recall $\out(\mathcal I_{\bd}(F)) = \omega_p$.
	By Lemma~\ref{Lemma - kappa, theta no tukey map}, we get that $F\not\leq_T E$, hence $F\not \leq_T C$ as sought.
\end{proof}

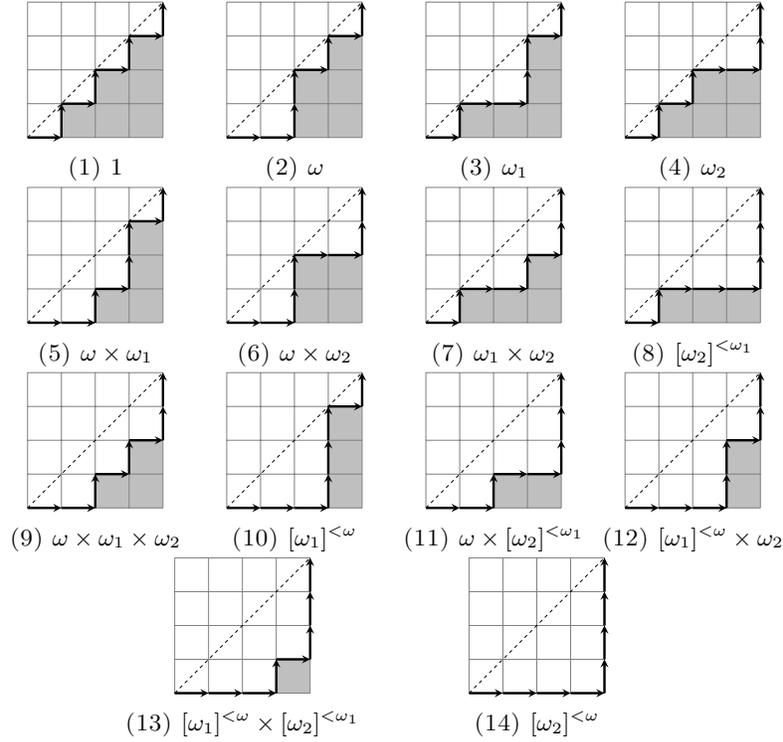
\begin{figure}[h]\label{Figure - good 4-paths}
	\begin{subfigure}[h]{0.20\textwidth}
		\adjustbox{scale=0.45, center}{\begin{tikzcd}
				\catalannumber{	0,0}{4}{0,1,0,1,0,1,0,1}
		\end{tikzcd}}
		\caption{$1$}
	\end{subfigure}
	\begin{subfigure}[h]{0.20\textwidth}
		\adjustbox{scale=0.45, center}{\begin{tikzcd}
				\catalannumber{	0,0}{4}{0,0,1,1,0,1,0,1}
		\end{tikzcd}}
		\caption{$\omega$}
	\end{subfigure}
	\begin{subfigure}[h]{0.20\textwidth}
		\adjustbox{scale=0.45, center}{\begin{tikzcd}
				\catalannumber{0,0}{4}{0,1,0,0,1,1,0,1}
		\end{tikzcd}}
		\caption{$\omega_1$}
	\end{subfigure}
	\begin{subfigure}[h]{0.20\textwidth}
		\adjustbox{scale=0.45, center}{\begin{tikzcd}
				\catalannumber{0,0}{4}{0,1,0,1,0,0,1,1}
		\end{tikzcd}}
		\caption{$\omega_2$}
	\end{subfigure}
	\begin{subfigure}[h]{0.20\textwidth}
		\adjustbox{scale=0.45, center}{\begin{tikzcd}
				\catalannumber{0,0}{4}{0,0,1,0,1,1,0,1}
		\end{tikzcd}}
		\caption{$\omega\times\omega_1$}
	\end{subfigure}
	\begin{subfigure}[h]{0.20\textwidth}
		\adjustbox{scale=0.45, center}{\begin{tikzcd}
				\catalannumber{0,0}{4}{0,0,1,1,0,0,1,1}
		\end{tikzcd}}
		\caption{$\omega\times\omega_2$}
	\end{subfigure}
	\begin{subfigure}[h]{0.20\textwidth}
		\adjustbox{scale=0.45, center}{\begin{tikzcd}
				\catalannumber{0,0}{4}{0,1,0,0,1,0,1,1}
		\end{tikzcd}}
		\caption{$\omega_1\times\omega_2$}
	\end{subfigure}
	\begin{subfigure}[h]{0.20\textwidth}
		\adjustbox{scale=0.45, center}{\begin{tikzcd}
				\catalannumber{0,0}{4}{0,1,0,0,0,1,1,1}
		\end{tikzcd}}
		\caption{$[\omega_2]^{<\omega_1}$}
	\end{subfigure}
	\begin{subfigure}[h]{0.20\textwidth}
		\adjustbox{scale=0.45, center}{\begin{tikzcd}
				\catalannumber{0,0}{4}{0,0,1,0,1,0,1,1}
		\end{tikzcd}}
		\caption{$\omega\times\omega_1\times\omega_2$}
	\end{subfigure}
	\begin{subfigure}[h]{0.20\textwidth}
		\adjustbox{scale=0.45, center}{\begin{tikzcd}
				\catalannumber{0,0}{4}{0,0,0,1,1,1,0,1}
		\end{tikzcd}}
		\caption{$[\omega_1]^{<\omega}$}
	\end{subfigure}
	\begin{subfigure}[h]{0.20\textwidth}
		\adjustbox{scale=0.45, center}{\begin{tikzcd}
				\catalannumber{0,0}{4}{0,0,1,0,0,1,1,1}
		\end{tikzcd}}
		\caption{$\omega\times [\omega_2]^{<\omega_1}$}
	\end{subfigure}
	\begin{subfigure}[h]{0.20\textwidth}
		\adjustbox{scale=0.45, center}{\begin{tikzcd}
				\catalannumber{0,0}{4}{0,0,0,1,1,0,1,1}
		\end{tikzcd}}
		\caption{$[\omega_1]^{<\omega}\times \omega_2$}
	\end{subfigure}
	\centering{
		\begin{subfigure}[h]{0.3\textwidth}
			\adjustbox{scale=0.45, center}{\begin{tikzcd}
					\catalannumber{0,0}{4}{0,0,0,1,0,1,1,1}
			\end{tikzcd}}
			\caption{$[\omega_1]^{<\omega}\times [\omega_2]^{<\omega_1}$}
		\end{subfigure}
		\begin{subfigure}[h]{0.30\textwidth}
			\adjustbox{scale=0.45, center}{\begin{tikzcd}
					\catalannumber{0,0}{4}{0,0,0,0,1,1,1,1}
			\end{tikzcd}}
			\caption{$ [\omega_2]^{<\omega}$}
		\end{subfigure}
	}
	\caption{All good $4$-paths and the corresponding types in $\mathcal T_2$ they encode.}
\end{figure}

\begin{theorem}
	The posets $(\mathcal T_n, <_T)$ and $(\mathcal K_{n+2},\vartriangleleft)$ are isomorphic.
\end{theorem}

\begin{proof}
	Define $f$ from $(\mathcal T_n,<_T)$ to $(\mathcal K_{n+2},\vartriangleleft)$, where for $C\in \mathcal T_n$, we let $f(C):=\vec v_C$.
	By Lemmas~\ref{Lemma - path < implies Tukey <} and~\ref{Lemma - path not < implies Tukey not <}, this is indeed an isomorphism of posets.
	
	Furthermore, we claim that $\mathcal T_n$ contains one unique representative from each equivalence class of $(\mathcal S_n,\equiv_T)$.
	Recall that the function $\mathfrak F$ is preserving Tukey equivalence classes.
	Consider two distinct $C,D\in \mathcal T_n$.
	As the coding is a bijection, $\vec v_C$ and $\vec v_D$ are different.
	Notice that either $\vec v_C \not \vartriangleleft \vec v_D$ or $\vec v_D \not \vartriangleleft \vec v_C$, hence by Lemma~\ref{Lemma - path not < implies Tukey not <}, $C\not \equiv_T D$ as sought.
\end{proof}

Consider the poset $(\mathcal T_n, <_T)$, clearly $1$ is a minimal element and by Lemma~\ref{Lemma - below kappa theta}, $[\omega_n]^{<\omega}$ is a maximal element.
By the previous Theorem, the set of immediate successors of an element $D$ in the poset $(\mathcal T_n, <_T)$, is the set of all directed sets $C\in \mathcal T_n$ such that $\vec v_C$ is an $\vartriangleleft $-immediate successor of $\vec v_D$.

\begin{lemma}
	Suppose $G,H\in \mathcal T_n$, $H$ is an immediate successor of $G$ in the poset $(\mathcal T_n, <_T)$ and $(\vec v_G, \vec v_H)$ are on the $l$-diagonal.
	Then there are $C, E, M, N$ directed sets such that:
	\begin{itemize}
		\item $G\equiv_T C\times M\times E$ and $H\equiv_T C\times N \times E$;
		\item for some $k\leq n$, $\cf(N)=\omega_k$, $|C|<\omega_k$ and either $E\equiv_T 1$ or $\non(\mathcal I_{\bd}(E))>\omega_{k-l}$.
	\end{itemize}

	Furthermore,
	\begin{itemize}
		\item If $l=0$, then $M=1$ and $N=\omega_k$.
		\item If $l=1$, then $k>1$ and $M=\omega_k$ and $N=[\omega_k]^{<\omega_{k-1}}$.
		\item If $l>1$, then $k>l$ and $M=[\omega_k]^{<\omega_{k-l+1}}$ and $N=[\omega_k]^{<\omega_{k-l}}$.
	\end{itemize}
\end{lemma}
\begin{proof}
	As $H$ is an immediate successor of $G$ in the poset $(\mathcal T_n, <_T)$, we know that $\vec v_H$ is an immediate successor of  $\vec v_H$ in $(\mathcal K_{n+2},\vartriangleleft)$.
	Let $k$ be the unique $k\leq n$ such that $v^k_G = v^k_H +1$.
	
	Let $\vec v_G:=\langle v^0_G, \dots, v^n_G \rangle$ be a good $(n+2)$-path coded by $G$.
	We construct $\langle M_i \mid i\leq n \rangle$ by letting $M_i$ be the $i$th element of the vector $\langle [\omega_i]^{<\omega},\dots, [\omega_i]^{<\omega_{i-1}} ,\omega_i, 1\rangle$ for every $i\leq n$.
	Notice that $G\equiv_T \prod_{i\leq n} M_i$.
	Similarly, we may construct $\langle N_i \mid i\leq n \rangle$ such that $H\equiv_T \prod_{i\leq n}	N_i$.
	Clearly, $M_i=N_i$ for every $i\neq k$.

	Let $C:=\prod_{i<k} M_i$ and $E:=\prod_{i>k} M_i$.
	Notice that $|C|=\cf(C)<\omega_k$ and either $E\equiv_T 1$ or $\non(\mathcal I_{\bd}(E))>\omega_{k-l}$.
	Moreover, $G\equiv_T C\times M_k \times E$ and $H\equiv_TC\times N_k \times E$.
	We split to cases:
	\begin{itemize}
		\item If $l=0$, then $v^k_H=k+1$, hence $M_k=1$ and $N_k=\omega_k$;
		\item If $l=1$, then $v^k_H=k$, hence $M_k=\omega_k$ and $N_k=[\omega_k]^{<\omega_{k-1}}$.

		\item If $l>1$, then $v^k_H=k-l+1$, hence $M_k=[\omega_k]^{<\omega_{k-l+1}}$ and $N_k=[\omega_k]^{<\omega_{k-l}}$.

	\end{itemize}
\end{proof}	

\begin{theorem}\label{Theorem - Intervals}
	Suppose $G,H\in \mathcal T_n$, $H$ is an immediate successor of $G$ in the poset $(\mathcal T_n, <_T)$ and $(\vec v_G, \vec v_H)$ are on the $l$-diagonal.
	\begin{itemize}
		\item If $l=0$, then there is no directed set $D\in \mathcal D_{\aleph_n}$ such that $G<_T D<_T H$;
		\item If $l>0$, then consistently there exist a directed set $D\in \mathcal D_{\aleph_n}$ such that $G<_T D<_T H$.
	\end{itemize}
\end{theorem}
\begin{proof}
	Let $C,E, M, N$ be as in the previous Lemma, so $G\equiv_T C\times M\times E$, $H\equiv_T C\times N \times E$ and for some $k\leq n$, $|C|<\omega_k$ and either $E\equiv_T 1$ or $\non(\mathcal I_{\bd}(E))>\omega_{k-l}$.
	We split to three cases:
	\begin{itemize}
		\item Suppose $l=0$, then $G\equiv_T C\times E $ and $H\equiv_T C\times \omega_k \times E$, by Theorem~\ref{Theorem - Gap} there is no directed set $D$ such that $G<_T D <_T H$.
		
		\item Suppose $l=1$, then $k\geq 1$ and $N=[\omega_k]^{<\omega_{k-1}}$ and $M=\omega_k$.
		\begin{itemize}
			\item Suppose $k=1$, then under the assumption $\mathfrak b =\omega_1$, by Theorem~\ref{Theorem - omega x omega_1 < D < [omega_1]^{<omega}} there exists a directed set $D$ such that $G<_T D <_T H$.
			\item Suppose $k>1$, then under the assumption $2^{\aleph_{k-2}} = \aleph_{k-1}$ and $2^{\aleph_{k-1}}= \aleph_k$, by Corollary~\ref{Cor - theta^+ x theta^++ < D < [theta^++]^{<theta+}} there exists a directed set $D$ such that $G<_T D <_T H$.
		\end{itemize}
		\item Suppose $l>1$, then $k\geq 2$.
		Let $\theta=\omega_{k-l}$ and $\lambda = \omega_{k-1}$.
		Notice $N=[\omega_k]^{\leq\theta}$ and $M=[\omega_k]^{<\theta}$.
		In Corollary~\ref{Corollary - directed set D_{mathcal C}} below, we shall show that under the assumption $\lambda^{\theta}<\lambda^+$ and $ \clubsuit^{\omega_{k-1}}_{J}(S,1) $ for some stationary set $S\subseteq E^{\omega_k}_{\theta}$, 
		there exists a directed set $D$ such that $G<_T D <_T H$.\qedhere
	\end{itemize}
\end{proof}

\section{Empty intervals in $D_{\aleph_n}$}\label{Section - Gaps}
Consider two successive directed sets in the poset $(\mathcal T_n , <_T)$, we can ask whether there exists some other directed set in between in the Tukey order.
The following Theorem give us a scenario in which there is a no such directed set.

\begin{theorem}\label{Theorem - Gap}
	Let $\kappa$ be a regular cardinal.
	Suppose $C$ and $E$ are two directed sets such that $\cf(C)<\kappa$ and either $E\equiv_T 1$ or $\kappa \in \Inner(\mathcal I_{\bd}(E),\kappa)$ and $\kappa\leq \cf(E)$.
	Then there is no directed set $D$ such that $C\times E <_T D <_T C\times \kappa\times E$.
\end{theorem}
\begin{proof} By the upcoming Lemmas \ref{lemma41} and \ref{Lemma - 2nd gap}.
\end{proof}

\begin{lemma}\label{lemma41}
	Let $\kappa$ be a regular cardinal.
	Suppose $C$ is a directed set such that $\cf(C)<\kappa$, then there is no directed set $D$ such that $C <_T D <_T C \times \kappa$.
\end{lemma}
\begin{proof}
	Suppose $D$ is a directed set such that $C <_T D <_T C \times \kappa$.
	Let us assume $D$ is a directed set of size $\cf(D)$ such that every subset of $D$ of size $\cf(D)$ is unbounded in $D$.
	By Lemma~\ref{Lemma - C leq_T D imply cf(C) leq cf(D)} we get that $\cf(C) \leq \cf(D)\leq \kappa$.
	We split to two cases:
	
	$\br$ Suppose $\cf(C)\leq \cf(D) <\kappa$.
	Let $g:D\rightarrow C\times \kappa$ be a Tukey function.
	As $|D|=\cf(D)<\kappa$ and $\kappa$ is regular there exists some $\alpha<\kappa$ such that $g"D \subseteq C\times \alpha$.
	We claim that $\pi_C \circ g$ is a Tukey function from $D$ to $C$, hence $D\leq_T C$ which is absurd.
	Suppose $X\subseteq D$ is unbounded in $D$, as $g$ is a Tukey function, we know that $g"X$ is unbounded in $C\times \kappa$.
	But as $(\pi_{\kappa}\circ g)"X$ is bounded by $\alpha$, we get that $(\pi_{C}\circ g)" X$ is unbounded in $C$ as sought.
	
	$\br$ Suppose $\cf(D)=\kappa$, notice that $\kappa \in \hu(D)$ is regular so by Lemma~\ref{Lemma - kappa leq_T D} we get that $\kappa\leq_T D$.
	We also know that $C\leq_T D$, thus $\kappa\times C \leq_T D$ which is absurd.
\end{proof}

Note that $\non(\mathcal I_{\bd}(E)) > \kappa$ implies that $\kappa \in \Inner(\mathcal I_{\bd}(E),\kappa)$.
\begin{lemma}\label{Lemma - 2nd gap}
	Let $\kappa$ be a regular cardinal.
	Suppose $C$ and $E$ are two directed sets such that $\cf(C)<\kappa\leq \cf(E)$ and $\kappa \in \Inner(\mathcal I_{\bd}(E),\kappa)$. 
	Then there is no directed set $D$ such that $C\times E <_T D <_T C\times \kappa\times E$.
\end{lemma}
\begin{proof}
	Suppose $D$ is a directed set such that $C\times E \leq_T D \leq_T C\times \kappa\times E$, we will show that either $D\equiv_T C\times E$ or $D\equiv_T C\times \kappa \times E$.
	We may assume that every subset of D of size $\cf(D)$ is unbounded and $|C|=\cf(C)$.
	By Lemma~\ref{Lemma - C leq_T D imply cf(C) leq cf(D)}, we have that $\cf(E)=\cf(D)$.
	
Suppose first there exists some unbounded subset $X\in [D]^{\kappa}$ such that every subset $Y\in [X]^{<\kappa}$ is bounded.
	By Corollary~\ref{Cor - existence of full-unbounded kappa set}, this implies that $\kappa\leq_T D$.
	But as $C\times E\leq_T D$ and  $D \leq_T C\times \kappa\times E$, this implies that $C\times \kappa \times E \equiv_T D$ as sought.
	
Hereafter, suppose for every unbounded subset $X\in [D]^{\kappa}$ there exists some subset $Y\in [X]^{<\kappa}$ unbounded.
	Let $g:D\rightarrow C\times \kappa\times E$ be a Tukey function. 
	Define $h:=\pi_{C\times E} \circ g$.
Now, there are two main cases to consider:		
	\begin{itemize}
		\item[$\br$] Suppose every unbounded subset $X\subseteq D$ of size $\lambda>\kappa$ which contain no unbounded subset of smaller cardinality is such that $h"X$ is unbounded in $C\times E$.
		
		We show that $h$ is Tukey, it is enough to verify that for every cardinal $\omega\leq \mu\leq \kappa$ and every unbounded subset $X\subseteq D$ of size $\mu$ which contain no unbounded subset of smaller cardinality is such that $h"X$ is unbounded in $C\times E$.
		
		As $g$ is Tukey, the set $g"X$ is unbounded in $C\times \kappa \times E$.
		Notice that if the set $\pi_{C\times E}\circ g" X$ is unbounded, then we are done.
		Assume that $\pi_{C\times E}\circ g" X$ is bounded, then $\pi_{\kappa} \circ g"X$ is unbounded.
		
		$\br\br$ Suppose $|X|<\kappa$. As $|g"X|<\kappa$, we have that $\pi_{\kappa} \circ g"X$ is bounded, which is absurd.
		
		$\br\br$ Suppose $ |X|=\kappa$, by the case assumption there exists some $Y \in [X]^{<\kappa}$ unbounded in $D$.
		But this is absurd as the assumption on $X$ was that $X$ contains no subset of size smaller than $|X|$ which is unbounded.	
		
		$\br\br$ Suppose $|X|>\kappa$, by the case assumption, $h"X$ is unbounded in $C\times E$ as sought.
		
		\item[$\br$] Suppose for some unbounded subset $X\subseteq D$ of size $\lambda>\kappa$ which contains no unbounded subset of smaller cardinality is such that $h"X$ is bounded in $C\times E$.
		 As $g$ is Tukey, $\pi_{\kappa}\circ g"X$ is unbounded.
		
		Let $X_\alpha := X\cap g^{-1} (C\times \{\alpha\}\times E )$ and $U_\alpha:=\bigcup_{\beta\leq \alpha} X_\beta$ for every $\alpha<\kappa$.
		As $g$ is Tukey and $g"U_\alpha$ is bounded, we get that $U_\alpha$ is also bounded by some $y_\alpha\in D$.
		Let $Y:=\{y_\alpha \mid \alpha<\kappa\}$.
		We claim that $Y$ is of cardinality $\kappa$.
		If it wasn't, then by the pigeonhole principle as $\kappa$ is regular there would be some $\alpha<\kappa$ such that $y_\alpha$ bounds the set $X$ in $D$ and that is absurd.
		Similarly, as $X$ is unbounded, the set $Y$ and also every subset of it of size $\kappa$ must be unbounded.
		
		Next, we aim to get $Z\in [Y]^\kappa$ such that $\pi_{C\times E}\circ g"Z$ is bounded by some $(c,e)\in C\times E$.
		This can be done as follows:
		As $|C|<\kappa$ and $\kappa$ is regular, by the pigeonhole principle, there exists some $Z_0\in [Y]^\kappa$ and $c\in C$ such that $g"Z_0 \subseteq \{c\}\times \kappa \times E$.
		Similarly, if $|\pi_{E}\circ g" Z_0|<\kappa$, by the pigeonhole principle, there exists some $Z\in [Z_0]^\kappa$ and $e\in E$ such that $g"Z \subseteq \{c\}\times \kappa \times \{e\}$.
		Else, if $|\pi_{E}" Z_0|=\kappa$, then as $\kappa \in\Inner(\mathcal I_{\bd}(E),\kappa )$ for some $B\in [\pi_{E}\circ g" Z_0]^{\kappa}$ and $e\in E$, $B$ is bounded in $E$ by $e$. 
		Fix some $Z\in [Z_0]^\kappa$ such that $g"Z\subseteq \{c\}\times \kappa \times B$.
		
		Note that $Z$ is a subset of $Y$ of size $\kappa$, hence, unbounded in $D$.
		By the assumption, there exists some subset $W\in [Z]^{<\kappa}$ unbounded in $D$.
		Note that as $\kappa$ is regular, for some $\alpha<\kappa$, $\pi_{\kappa} \circ g"W\subseteq \alpha$.
		As $g$ is Tukey, the subset $g"W\subseteq \{c\}\times \kappa \times E$ is unbounded in $C\times \kappa \times E$, but this is absurd as $g"W$ is bounded by $(c,\alpha,e)$.
	\qedhere
	\end{itemize}
\end{proof}

\section{Non-empty intervals}\label{Section - No Gaps}

In this section we consider three types of intervals in the poset $(\mathcal T_n, <_T)$ and show each one can consistently have a directed set inside.

\subsection{Directed set between $  \theta^+ \times \theta^{++}$ and $[\theta^{++}]^{\leq \theta}$}

In {\cite[Theorem~1.1]{kuzeljevic2021cofinal}}, the authors constructed a directed set between  $\omega_1\times \omega_2$ and $[\omega_2]^{\leq \omega}$ under the assumption $2^{\aleph_0}=\aleph_1$, $2^{\aleph_1}=\aleph_2$ and the existence of an $\aleph_2$-Souslin tree.
In this subsection we generalize this result while waiving the assumption concerning the Souslin tree.
The main corollary of this subsection is:
\begin{cor}\label{Cor - theta^+ x theta^++ < D < [theta^++]^{<theta+}}
	Assume $\theta$ is an infinite cardinal such that $2^\theta = \theta^+$, $2^{\theta^+}= \theta^{++}$.
	Suppose $C$ and $E$ are directed sets such that $\cf(C)\leq \theta^+$ and either $\non(\mathcal I_{\bd} (E))>\theta^+$ or $E\equiv_T 1$.
	Then there exists a directed set $D$ such that $ C\times \theta^+ \times \theta^{++}\times E <_T C\times D\times E <_T C\times[\theta^{++}]^{\leq \theta}\times E.$
\end{cor}

The result follows immediately from Theorems~\ref{theorem - D between theta x times+ and [theta++]<=theta} and \ref{Theorem - directed set from coloring}.
First we prove the following required Lemma.

\begin{lemma}\label{Lemma - not <_T cofinal unbounded >theta and >theta bounded}
	Suppose $\theta$ is a infinite cardinal and $D,J, E$ are three directed sets such that:
	\begin{itemize}
		\item $\cf(D)=\cf(J)=\theta^{++}$;
		\item $\theta^+ \in \Inner(\mathcal I_{\bd}(D), \theta^{++})$ and $\out(\mathcal I_{\bd}(J))\leq\theta^+$;
		\item $\non(\mathcal I_{\bd}(E))>\theta^+$ or $E\equiv_T 1$.
		\item $D\times E\leq_T J\times E$.
	\end{itemize}
	Then $J\times E\not \leq_T D\times E$. 
	In particular, $D\times E <_T J\times E$.
\end{lemma}
\begin{proof}
	Notice that $D$ is a directed set such that every subset of size $\theta^{++}$ contains a bounded subset of size $\theta^+$.
	Let us fix a cofinal subset $A\subseteq J$ of size $\theta^{++}$ such that every subset of $A$ of size $>\theta$ is unbounded in $J$.
	
	Suppose on the contrary that $J\times E\leq_T D\times E$.
	As $D\times E\leq_T J\times E$ we get that $D\times E\equiv_T J\times E$, hence there exists some directed set $X$ such that both $D\times E$ and $J\times E$ are cofinal subsets of $X$.
	
	We may assume that $D$ has an enumeration $D:=\{d_\alpha \mid \alpha<\theta^{++}\}$ such that for every $\beta<\alpha<\theta^{++}$ we have $d_\alpha \not <d_\beta$.
	Fix some $e\in E$.
	Now, for each $a\in A$ take a unique $x_a\in D$ and some $e_a\in E$ such that $(a,e)\leq_X (x_a,e_a)$.
	To do that, enumerate $A=\{a_\alpha \mid \alpha<\theta^{++}\}$. Suppose we have constructed already the increasing sequence 
	$\langle \nu_\beta \mid \beta<\alpha \rangle$ of elements in $\theta^{++}$.
	Pick some $\xi<\theta^{++}$ above $\{ \nu_\beta \mid \beta<\alpha \}$.
	As $D\times E$ is a directed set we may fix some $(x_{a_\alpha},e_a):=(d_{\nu_\alpha},e_a)\in D\times E$ above $(a_\alpha,e)$ and $(d_\xi,e)$. 
	
	Set $T=\{x_a\mid a\in A\}$, since $A\times E$ is cofinal in $X$, the set $T\times E$  is also cofinal in $X$ and $D\times E$.
	As $|T|=\theta^{++}$ we get that there exists some subset $B\in [T]^{\theta^+}$ bounded in $D$.
	Let $c\in D$ be such that $b\leq c$ for each $b\in B$.
	Consider the set $K=\{a\in A \mid x_a \in B \}$.
	Since either $\non(\mathcal I_{\bd}(E))>\theta^+$ or $E\equiv_T 1$, as $\{e_a\mid a\in K\}$ is of size $\leq\theta^+$, it is bounded in $E$ by some $\tilde e\in E$.
	So $P:=\{(x_a, e_a)\mid a\in K\}$ is bounded in $X$.
	Since $B$ is of size $>\theta$, the set $K$ is also of size $>\theta$.
	Thus, by the assumption on $A$, the set $K\times\{e\}$ is unbounded in $J\times E$, but also in $X$ because $J\times E$ is a cofinal subset of $X$.
	Then, for each $a\in K$ we have $(a,e)\leq_X  (x_a,e_a) \leq_X (c,\tilde e)$, contradicting the unboundness of $K\times \{e\}$ in $J\times E$.
\end{proof}

\begin{theorem}\label{theorem - D between theta x times+ and [theta++]<=theta}
 	Suppose $\theta$ is an infinite cardinal and $C,D,E$ are directed sets such that:
	\begin{enumerate}[label=(\arabic*)]
		\item $\cf(D)=\theta^{++}$;
		\item\label{Clause - unbounded set of size theta}
		For every partition $D=\bigcup_{\gamma<\theta^+} D_\gamma$, there is an ordinal $\gamma<\theta^+$, and an unbounded $K\subseteq D_\gamma$ of size $\theta^+$;
		
		\item\label{Clause - bounded lambda set}$\theta^+ \in \Inner(\mathcal I_{\bd}(D), \theta^{++})$ and $\non(\mathcal I_{\bd}(D))=\theta^+$;
		\item  $\non(\mathcal I_{\bd}(E))>\theta^+$ or $E\equiv_T 1$;
		\item $C$ is a directed set such that $\cf(C)\leq\theta^+$.
	\end{enumerate}
	Then $ C\times \theta^+ \times \theta^{++}\times E <_T C\times D\times E <_T C\times[\theta^{++}]^{\leq \theta}\times E.$
\end{theorem}
\begin{proof} As $\cf(D)=\theta^{++}$, we may assume that every subset of $D$ of size $\theta^{++}$ is unbounded.
	\begin{claim}
		$\theta^+\times \theta^{++}\leq_T D$.
	\end{claim}
	\begin{subproof}
		As $\cf(D)=\theta^{++}$, we get by Lemma~\ref{Lemma - kappa leq_T D} that $\theta^{++}\leq_T D$.
		Let $K$ be an unbounded subset of $D$ of size $\theta^+$, as every subset of size $\theta$ is bounded, by Corollary~\ref{Cor - existence of full-unbounded kappa set} we get that $\theta^+\leq_T D$.
		Hence, $\theta^+\times \theta^{++} \leq_T D$ as sought.
	\end{subproof}
	\begin{claim}
		$D \leq_T [\theta^{++}]^{\leq \theta}$.
	\end{claim}
	\begin{subproof}
		As $\cf(D)=\theta^{++}$ and $\non(\mathcal I_{\bd}(D))=\theta^+$, by Lemma~\ref{Lemma - below kappa theta}, $D \leq_T [\theta^{++}]^{\leq \theta}$ as sought.
	\end{subproof}
	Notice this implies that  $  C\times \theta^+ \times \theta^{++}\times E \leq _T  C\times  D\times E \leq_T  C\times [\theta^{++}]^{\leq \theta}\times E.$

	By Lemma~\ref{Lemma - D not<= C x E}, as $|C\times \theta^+|\leq \theta^+$, $\non(\mathcal I_{\bd}(\theta^{++}\times E))>\theta^+$ and Clause~\ref{Clause - unbounded set of size theta} we get that $D\not\leq_T C\times \theta^+\times \theta^{++}\times E$.
	\begin{claim}
		$C\times [\theta^{++}]^{\leq \theta}\times E \not \leq_T C\times D\times E$.
	\end{claim}
	\begin{subproof}
		Recall that $\mathfrak D_{ [\theta^{++}]^{\leq \theta}}\equiv_T  [\theta^{++}]^{\leq \theta}$. 
		Notice that following:
		\begin{itemize}
			\item $\cf(C\times D)=\cf(C\times \mathfrak D_{[\theta^{++}]^{\leq\theta}})=\theta^{++}$;
			\item By Clause~\ref{Clause - bounded lambda set} we have that $\theta^+ \in \Inner(\mathcal I_{\bd}(C\times D), \theta^{++})$ and 
			$\out(\mathcal I_{\bd}(C\times \mathfrak D_{ [\theta^{++}]^{\leq \theta}}))\leq \theta^{+}$;
			\item $\non(\mathcal I_{\bd}(E))>\theta^+$ or $E=1$;

			\item $C\times D\times E\leq_T C\times \mathfrak D_{[\theta^{++}]^{\leq\theta}}\times E$.
		\end{itemize}
	So by Lemma~\ref{Lemma - not <_T cofinal unbounded >theta and >theta bounded} we are done.\qedhere
	\end{subproof}\qedhere
\end{proof}

We are left with proving the following Theorem, in which we define a directed set $D_c$ using a coloring $c$.

\begin{theorem}\label{Theorem - directed set from coloring}
	Suppose $\theta$ is an infinite cardinal such that $2^\theta= \theta^+$ and $2^{\theta^+}=\theta^{++}$. Then there exists a directed set $D$ such that:
	\begin{enumerate}[label=(\arabic*)]
		\item $\cf(D)=\theta^{++}$.
		\item For every partition $D=\bigcup_{\gamma<\theta^+} D_\gamma$, there is an ordinal $\gamma<\theta^+$, and an unbounded $K\subseteq D_\gamma$ of size $\theta^+$;
		\item $\theta^+ \in \Inner(\mathcal I_{\bd}(D), \theta^{++})$ and $\non(\mathcal I_{\bd}(D))=\theta^+$.
	\end{enumerate}
\end{theorem}

The rest of this subsection is dedicated  to proving Theorem~\ref{Theorem - directed set from coloring}. The arithmetic hypothesis will only play a role later on.
Let $\theta$ be an infinite cardinal.
For two sets of ordinals $A$ and $B$, we denote $A\circledast B:= \{(\alpha,\beta )\in A\times B\mid \alpha<\beta\}$.
Recall that by \cite[Corollary~7.3]{paper47}, $\onto(\mathcal S,J^{\bd}[\theta^{++}],\theta^+)$ holds
for $\mathcal S:=[\theta^{++}]^{\theta^{++}}$. This means that we may fix a coloring $c:[\theta^{++}]^2\rightarrow \theta^+$ such that for every $S\in\mathcal S$ and unbounded $B\s\theta^{++}$, there exists $\delta\in S$ such that $c"(\{\delta\}\circledast B)=\theta^+$.

We fix some $S\in\mathcal S$. For our purpose, it will suffice to assume that $S$ is nothing but the whole of $\theta^{++}$.
Let $$D_{c}:=\{ X\in [\theta^{++}]^{\leq \theta^{+}} \mid \forall \delta\in S[ \{c(\delta,\beta) \mid \beta\in X\setminus(\delta+1)\}\in \ns_{\theta^+} ]\}.$$

Consider $D_{c}$ ordered by inclusion, and notice that $D_{c}$ is a directed set since $\ns_{\theta^+}$ is an ideal.

\begin{prop} The following hold:
	\begin{itemize}
		\item $[\theta^{++}]^{\leq\theta} \subseteq D_{c} \subseteq [\theta^{++}]^{\leq \theta^+}$.
		\item $\non(\mathcal I_{\bd}(D_c))\geq\add(\mathcal I_{\bd}(D_c))\geq\theta^+$, i.e. every family of bounded subsets of $D_{c}$ of size $< \theta^+$ is bounded.
		\item If $2^{\theta^+}=\theta^{++}$, then $|D_c|=\theta^{++}$, and hence $D_c\in\mathcal D_{\theta^{++}}$.\qed
	\end{itemize}
\end{prop}

\begin{lemma}\label{D_C - Lemma 5.3}
	For every partition $D_{c}=\bigcup_{\gamma<\theta^+} D_\gamma$, there is an ordinal $\gamma<\theta^+$, and an unbounded $E\subseteq D_\gamma$ of size $\theta^+$.
\end{lemma}
\begin{proof}
	As $[\theta^{++}]^1$ is a subset of $D_{c}$, the family $\{D_{\gamma }\mid \gamma<\theta^+\}$ is a partition of the set $[\theta^{++}]^1$ to at most $\theta^+$ many sets. As $\theta^+<\theta^{++}=\cf(\theta^{++})$, by the pigeonhole principle we get that for some $\gamma<\theta^+$ and $b\in [\theta^{++}]^{\theta^{++}}$, we have $[b]^1 \subseteq D_{\gamma}$.
	Notice that by the assumption on the coloring $c$, there exists some $\delta\in S$ and $\delta<b'\in [b]^{\theta^+}$ such that $c"(\delta \circledast b')=\theta^+$.
	Clearly the set $E:=[b']^1$ is a subset of $ D_{\gamma}$ of size $\theta^+$ which is unbounded in $D_{c}$.
\end{proof}

\begin{lemma}\label{D_C - Lemma 5.4}
	Suppose $2^{\theta}=\theta^+$, then $\theta^+ \in  \Inner(\mathcal I_{\bd}(D_{c}),\theta^{++})$.
\end{lemma}
\begin{proof} We follow the proof of \cite[Lemma~5.4]{kuzeljevic2021cofinal}.

	Let $D'$ be a subset of $D_{c}$ of size $\theta^{++}$ we will show it contains a bounded subset of size $\theta^+$, let us enumerate it as $\{T_\gamma \mid \gamma <\theta^{++}\}$.

	Let, for each $X\in D_{c}$ and $\gamma\in S$, $N^X_\gamma$ denote the non-stationary set $\{c(\gamma,\beta)\mid \beta\in X\setminus(\gamma+1)\}$, and let $G^X_\gamma$ denote a club in $\theta^+$ disjoint from $N^X_\gamma$.
	
	As $2^{\theta}=\theta^+$ we may fix a sufficiently large regular cardinal $\chi$, and an elementary submodel $M\prec H_{\chi}$ of cardinality $\theta^+$ containing all the relevant objects and such that $M^{\theta} \subseteq M$.
	Denote $\delta=M\cap \theta^{++}$, notice $\delta \in E^{\theta^{++}}_{\theta^+}$.
	Fix an increasing sequence $\langle \gamma_\xi \mid \xi <\theta^+\rangle$ in $\delta$ such that $\sup\{\gamma_\xi \mid \xi<\theta^+\}=\delta$.
	Enumerate $\delta\cap S=\{s_\xi \mid \xi<\theta^+\}$.
	In order to simplify notation, let $G^\gamma_\xi$ denote the set $G^{T_\gamma}_{s_\xi}$ for each $\gamma<\theta^{++}$ and $\xi<\theta^+$.
	
	We construct by recursion on $\xi<\theta^+$ three sequences $\langle \delta_\xi \mid \xi<\theta^+\rangle$,  $\langle \Gamma_\xi \mid \xi<\theta^+\rangle$ and  $\langle \eta_\xi \mid \xi<\theta^+\rangle$ with the following properties:
	\begin{enumerate}[label=(\arabic*)]
		\item $\langle \delta_\xi \mid \xi<\theta^+\rangle$ is an increasing sequence converging to $\delta$;
		\item $\langle \Gamma_\xi\mid \xi<\theta^+ \rangle$ is a decreasing $\subseteq$-chain of stationary subsets of $\theta^{++}$ each one containing $\delta$ and definable in $M$;
		\item $\langle \eta_\xi \mid \xi<\theta^+ \rangle$ is an increasing sequence of ordinals below $\theta^+$;
		\item $G^\delta_{\zeta} \cap \eta_{\mu} = G^{\delta_{\mu}}_{\zeta}\cap \eta_{\mu} $ for $\zeta\leq \mu<\theta^+$.
	\end{enumerate}
	
	$\br$ Base case: Let $\eta_0$ be the first limit point of $G^\delta_0$.
	Notice that $G^\delta_0\cap \eta_0$ is an infinite set of size $\leq\theta$ below $\delta$, hence it is inside of $M$.
	Let $$\Gamma_0:=\{\gamma<\theta^{++} \mid G^\delta_0\cap \eta_0 = G^\gamma_0\cap \eta_0\}.$$
	Since $\delta\in \Gamma_0$, the set $\Gamma_0$ is stationary in $\theta^{++}$.
	Let $\delta_0:=\min(\Gamma_0)$.
	
	$\br$ Suppose $\xi_0<\theta^+$, and that $\delta_\xi$, $\Gamma_\xi$ and $\eta_\xi$ have been constructed for each $\xi<\xi_0$.
	Let $\eta_{\xi_0}$ be the first limit point of $G^\delta_{\xi_0}\setminus \sup\{\eta_\xi \mid \xi<\xi_0\}$.
	Consider the set 
	$$ \Gamma_{\xi_0}=\large \{\gamma \in \bigcap_{\xi<\xi_0}\Gamma_\xi \mid \forall \xi\leq \xi_0 [G^\delta_\xi\cap \eta_{\xi_0} = G^\gamma_{\xi}\cap \eta_{\xi_0}]\large \}.$$
	Since $\Gamma_{\xi_0}$ belongs to $M$, and since $\delta\in \Gamma_{\xi_0}$, it must be that $\Gamma_{\xi_0}$ is stationary in $\theta^{++}$.
	Since $\Gamma_{\xi_0}$ is cofinal in $\theta^{++}$ and belongs to $M$, the set $\delta\cap \Gamma_{\xi_0}$ is cofinal in $\delta$. 
	Define $\delta_{\xi_0}$ be the minimal ordinal in $\delta\cap \Gamma_{\xi_0}$ greater than both $\sup\{\delta_\xi \mid \xi<\xi_0\}$ and $\gamma_{\xi_0}$.
	It is clear from the construction that conditions $(1-4)$ are satisfied.
	
	The following claim gives us the wanted result.
	\begin{claim}
		The set $\{T_{\delta_\xi}\mid \xi<\theta^+\}$ is a subset of $D'$ of size $\theta^+$ which is bounded in $D_{c}$.
	\end{claim}
	\begin{subproof}
		As the order on $D_{c}$ is $\subseteq$, it suffices to prove that the union $T=\bigcup_{\xi<\theta^+} T_{\delta_\xi}\in D_{c}$.
		Since, for each $\xi<\theta^+$, both $\delta_\xi$ and $\langle T_\gamma \mid \gamma<\theta^{++} \rangle$ belong to $M$, it must be that $T_{\delta_\xi}\in M$.
		Since $\theta^+\in M$ and $M\models |T_{\delta_\xi}|\leq \theta^+$, we have $T_{\delta_\xi}\subseteq M$.
		Thus $T\subseteq M$ and furthermore $T\subseteq \delta$.
		This means that, in order to prove that $T\in D_{c}$, it is enough to prove that for each $t \in S\cap \delta$, the set $\{c(t,\beta)\mid \beta\in T\setminus(t+1)\}$ is non-stationary in $\theta^+$.
		Fix some $t\in S\cap \delta$.
		Let $\zeta<\theta^+$ be such that $s_\zeta = t$.
		Define 
		$$ G:=G^\delta_\zeta \cap (\bigcap_{\xi\leq \zeta} G^{\delta_\xi}_\zeta)\cap (\triangle_{\xi<\theta^+} G^{\delta_\xi}_\zeta).$$
		Since the intersection of $< \theta^+$-many clubs in $\theta^+$ is a club, and since diagonal intersection of $\theta^+$ many clubs is a club, we know that $G$ is a club in $\theta^+$.
		
		We will prove that $G\cap \{c(t,\beta)\mid \beta\in T\setminus(t+1)\}=\emptyset$.
		Suppose $\alpha<\theta^+$ is such that $\alpha\in G \cap  \{c(t,\beta)\mid \beta\in T\setminus(t+1)\}$.
		This means that $\alpha\in G$ and that for some $\mu<\theta^+$ and $\beta\in T_{\delta_\mu}\setminus(t+1)$ we have $\alpha = c(t,\beta)$.
		So $\alpha \in N^{T_{\delta_\mu}}_{t}$.
		Note that this implies that $\alpha\notin G^{\delta_\mu}_\zeta$.
		Let us split to three cases:
		
		$\br$ Suppose $\mu\leq \zeta$, then since $\alpha\in \bigcap_{\xi\leq \zeta} G^{\delta_\xi}_\zeta$, we have that $\alpha\in G^{\delta_\mu}_\zeta$ which is clearly contradicting $\alpha\notin G^{\delta_\mu}_\zeta$.
		
		$\br$ Suppose $\mu>\zeta$ and $\alpha<\eta_\mu$.
		Then by $(4)$, we have that $G^\delta_\zeta \cap \eta_\mu =G^{\delta_\mu}_\zeta \cap \eta_\mu$.
		As $\alpha \not \in  G^{\delta_\mu}_\zeta$ and $\alpha<\eta_\mu$, it must be that $\alpha\notin G^\delta_\zeta$.
		Recall that $\alpha \in G$, but this is absurd as $G\subseteq G^\delta_\zeta$ and $\alpha\notin G^\delta_\zeta$.
		
		$\br$ Suppose $\mu>\zeta$ and $\alpha\geq \eta_\mu\geq\mu$.
		As $\alpha\in G$, we have that $\alpha\in \triangle_{\xi<\theta^+} G^{\delta_\xi}_\zeta$.
		As $\alpha>\mu$, we get that $\alpha\in G^{\delta_\mu}_\zeta$ which is clearly contradicting $\alpha\notin G^{\delta_\mu}_\zeta$.\qedhere
	\end{subproof}\qedhere
\end{proof}

\subsection{Directed set between $\omega\times\omega_1$ and $[\omega_1]^{<\omega}$}
As mentioned in \cite{MR3247063}, by the results of Todor\v{c}evi\'{c} \cite{MR980949}, it follows that under the assumption $\mathfrak b=\omega_1$ there exists a directed set of size $\omega_1$ between the directed sets $\omega\times\omega_1$ and $[\omega_1]^{<\omega}$.
In this subsection we spell out the details of this construction.

For two functions $f,g\in {}^{\omega}\omega$, we define the order $<^*$ by $f<^*g$ iff the set $\{n<\omega \mid g(n)\geq f(n)\}$ is finite.
Furthermore, by $f\lhd g$ we means that there exists $m<\omega$ such that for all $n<m$ we have $f(n)\leq g(n)$ and $f(k)<g(k)$ whenever $m\leq k<\omega$.
Assuming $f\leq^* g$, we let $\Delta(f,g):=\min\{m<\omega \mid \forall n\geq m [f(n)\leq g(n)]\}$.

The following fact is a special case of {\cite[Theorem~1.1]{MR980949}} in the case $n=0$, for complete details we give the proof as suggested by the referee.
\begin{fact}[Todor\v{c}evi\'{c}, {\cite[Theorem~1.1]{MR980949}}]\label{Fact - b=omega_1}
	Suppose $A$ is an uncountable sequence of ${}^{\omega}\omega$ of increasing functions which are $<^*$-increasing and $\leq^*$-unbounded, then there are $f,g\in A$ such that $f\lhd g$.
\end{fact}
\begin{proof}
	Let $A:=\{g_\alpha\mid \alpha<\omega_1\}$ be an uncountable sequence of increasing functions of ${}^{\omega}\omega$ which are $<^*$-increasing and $\leq^*$-unbounded.
	
	Let us fix a countable elementary sub-model $M\prec (H_{\omega_2},\in)$ with $A\in M$.
	Let $\delta:=\omega_1\cap M$, $B:=\omega_1\setminus(\delta+1)$ and write $B_n:=\{\beta\in B\mid \Delta(g_\delta,g_\beta)=n \}$.
	As $B=\bigcup_{n<\omega}B_n$, let us fix some $n<\omega$ such that $B_n$ is uncountable.
	As $\{g_\alpha\mid \alpha\in B_n\}$ is unbounded, we get that the set $K:=\{m<\omega \mid \sup\{g_\beta(m)\mid \beta\in B_n\} = \omega\}$ is non-empty, so consider the minimal element, $m:=\min(K)$.
	For $t\in{}^{m}\omega$, denote $B^t_n:=\{\beta\in B_n \mid t\subseteq g_\beta\}$.
	By minimality of $m$, the set $\{ t\in {}^{m}\omega \mid B^t_n \neq \emptyset \}$ is finite, so we can easily find some $t\in{}^{m}\omega$ such that $\sup\{  g_\beta(m)\mid \beta\in B^t_n\}=\omega$.
	
	Note that the set $\{\beta< \omega_1 \mid t\subseteq g_\beta\}$ is a non-empty set that is definable from $A$ and $t$, hence it is in $M$.
	Let us fix some $\alpha\in M\cap \omega_1$ such that $t\subseteq g_\alpha$.
	Put $k:=\Delta(g_\alpha,g_\delta)$, and then pick $\beta\in B^t_n$ such that $g_\beta(m)>g_\alpha(k+n)$. Of course, $\alpha<\delta<\beta$.
	We claim that $g_\alpha \vartriangleleft g_\beta$ as sought.
	
	Let us divide to three cases:
	\begin{itemize}
		\item If $i<m$, then $g_\alpha(i)=t(i)=g_\beta(i)$;
		\item If $m\leq i\leq k+n$, then $g_\alpha(i)\leq g_\alpha(k+n)<g_\beta(m)\leq g_\beta(i)$ recall that every function in $A$ is increasing;
		\item If $k+n<i<\omega$, then $\Delta(g_\alpha,g_\delta)=k<i$ and $g_\alpha(i)\leq g_\delta(i)$, as well as $\Delta(g_\delta,g_\beta)=n<i$ and $g_\delta(i)\leq g_\beta(i)$.
		Altogether, $g_\alpha(i)\leq g_\beta(i)$.\qedhere
	\end{itemize}\qedhere
\end{proof}

\begin{theorem}\label{Theorem - omega x omega_1 < D < [omega_1]^{<omega}}
	Assume $\mathfrak b =\omega_1$.
	Suppose $E$ is a directed set such that $\non(\mathcal I_{\bd} (E))>\omega$ or $E\equiv_T 1$.
	Then there exists a directed set $D$ such that: $$\omega\times \omega_1 \times E <_T D\times E <_T [\omega_1]^{<\omega}\times E.$$
\end{theorem}
\begin{proof}

Let $\mathcal F:=\langle f_\alpha \mid \alpha<\omega_1 \rangle\subseteq {}^{\omega}\omega$ witness $\mathfrak b=\omega_1$.
Recall $\mathcal F$ is a $<^*$-increasing and unbounded sequence, i.e. for every $g\in {}^{\omega}\omega$, there exists some $\alpha<\omega_1$ such that $f_\beta\not\leq^* g$, whenever $\alpha<\beta<\omega_1$.

For a finite set of functions $F\subseteq {}^{\omega}\omega$, we define a function $h:=\max(F)$ which is $\lhd$-above every function in $F$ by letting $h(n):=\max\{f(n)\mid f\in F\}$.
We consider the directed set $D:=\{ \max(F)\mid F\subseteq \mathcal F,~ |F|<\aleph_0 \}$, ordered by the relation $\lhd$, clearly $D$ is a directed set.

\begin{claim}\label{Claim - D_b every uncountable set contains infinite unbounded subset}
	Every uncountable subset $X\subseteq D$ contains a countable $B\subset X$ which is unbounded in $D$.
\end{claim}
\begin{subproof}
	Let $X$ be an uncountable subset of $D$.
	As $\mathcal F$ is a $<^*$-increasing and unbounded, also $X$ contains an uncountable $<^*$-unbounded subset $Y\subseteq X$.
	As no function $g:\omega\rightarrow \omega$ is $<^*$-bounding the set $Y$, we can find an infinite countable subset $B\subseteq Y$ and $n<\omega$ such that $\{f(n)\mid f\in B \}$ is infinite.
	Clearly $B$ is $\lhd$-unbounded in $D$ as sought.
\end{subproof}

\begin{claim}\label{Claim - D_b every uncountable set contains infinite bounded subset}
	$\omega \in \Inner(\mathcal I_{\bd}(D),\omega_1) $.
\end{claim}
\begin{subproof}
	We show that every uncountable subset of $D$ contains a countable infinite bounded subset.
	Let $A\subseteq D$ be an uncountable set, we may refine $A$ and assume that it is $<^*$-increasing and unbounded.
	We enumerate $A:=\{g_\alpha\mid \alpha<\omega_1\}$ and define a coloring $c:[\omega_1]^2\rightarrow 2$, letting for $\alpha<\beta<\omega_1$ the color $c(\alpha,\beta)=1$ iff $g_\alpha\lhd g_\beta$.
	Recall that Erd\"os and Rado showed that $\omega_1\rightarrow (\omega_1,\omega+1)^2$, so either there is an uncountable homogeneous set of color $0$ or there exists an homogeneous set of color $1$ of order-type $\omega+1$.
	Notice that Fact~\ref{Fact - b=omega_1} contradicts the first alternative, so the second one must hold.
	Let $X\subseteq \omega_1$ be a set such that $\otp(X)=\omega+1$ and $c"[X]^2=\{1\}$, notice that $\{g_\alpha \mid \alpha\in X\}$ is an infinite countable subset of $A$ which is $\lhd$-bounded by the function $g_{\max(X)}\in A$ as sought.\qedhere
\end{subproof}

Note that $\cf(D)=\omega_1$, hence $D\times E \leq_T [\omega_1]^{<\omega}\times E$.
	\begin{claim}
		$\omega\times \omega_1\times E\leq_T D\times E$.
	\end{claim}
	\begin{subproof}
		As every subset of $D$ of size $\omega_1$ is unbounded, we get by Lemma~\ref{Lemma - kappa leq_T D} that $\omega_1\leq_T D$.
		As $D$ is a directed set, every finite subset of $D$ is bounded.
		By Claim~\ref{Claim - D_b every uncountable set contains infinite unbounded subset}, $D$ contains an infinite countable unbounded subset, so by Corollary~\ref{Cor - existence of full-unbounded kappa set} we have $\omega \leq_T D$.
		Finally, $\omega\times \omega_1 \leq_T D$ as sought.
	\end{subproof}

	\begin{claim}
		$D\not \leq_T\omega \times E$.
	\end{claim}
	\begin{subproof}
		Recall that either $\non(\mathcal I_{\bd} (E))>\omega$ or $E\equiv_T 1$.
		Note that if $E\equiv_T 1$, then as $\cf(D)=\omega_1>\cf(\omega)$, we have by Lemma~\ref{Lemma - C leq_T D imply cf(C) leq cf(D)} that $D\not \leq_T\omega \times E$ as sought.
		Note that for every partition $D=\bigcup\{D_n\mid n<\omega\}$ of $D$, there exists some $n<\omega$ such that $D_n$ is uncountable, and by Claim~\ref{Claim - D_b every uncountable set contains infinite unbounded subset}, there exists some $X\subseteq D_n$ infinite and unbounded in $D$.
		As $\non(\mathcal I_{\bd} (E))>\omega$, by Lemma~\ref{Lemma - D not<= C x E} we have $D\not \leq_T\omega \times E$ as sought.
	\end{subproof}

	\begin{claim}
		$[\omega_1]^{<\omega} \not \leq_T D\times E$.
	\end{claim}
	\begin{subproof}
		By Claim~\ref{Claim - D_b every uncountable set contains infinite bounded subset}, every uncountable subset of $D$ contains an infinite countable bounded subset and every countable subset of $E $ is bounded, we get that $\omega \in \Inner(\mathcal I_{\bd}(D\times E),\omega_1)$.
		As $\out(\mathcal I_{\bd}([\omega_1]^{<\omega}))=\omega$ by Lemma~\ref{Lemma - kappa, theta no tukey map} we get that $[\omega_1]^{<\omega} \not \leq_T D\times E$ as sought.
	\end{subproof}\qedhere
\end{proof}

\subsection{Directed set between $[\lambda]^{<\theta}\times [\lambda^+]^{\leq \theta}$ and $[\lambda^+]^{< \theta}$}

In {\cite[Theorem~1.2]{kuzeljevic2021cofinal}}, the authors constructed a directed set between $[\omega_1]^{<\omega}\times [\omega_2]^{\leq \omega}$ and $[\omega_2]^{< \omega}$ under the assumption $2^{\aleph_0}=\aleph_1$, $2^{\aleph_1}=\aleph_2$ and the existence of a non-reflecting stationary subset of $E^{\omega_2}_\omega$.
In this subsection we generalize this result while waiving the assumption concerning the non-reflecting stationary set.

We commence by recalling some classic guessing principles and introducing a weak one, named $ \clubsuit^\mu_{J}(S,1) $, which will be useful for our construction. 

\begin{definition}\label{principles} For a stationary subset $ S\subseteq \kappa $:
	\begin{enumerate}
		\item $ \diamondsuit(S) $ asserts the existence of a sequence $ \langle C_\alpha \mid \alpha\in S \rangle $ such that:
		\begin{itemize}
			\item for all $ \alpha\in S $, $ C_\alpha \subseteq \alpha $;
			\item for every $B\s \kappa$, the set  $\{\alpha\in S\mid B\cap\alpha=C_\alpha \}$ is stationary.
		\end{itemize}
		\item 	$ \clubsuit(S) $ asserts the existence of a sequence $ \langle C_\alpha\mid \alpha\in S \rangle $ such that:
		\begin{itemize}
			\item\label{Definiton clubsuit - Clause A_alpha} for all $ \alpha\in S\cap \acc(\kappa) $, $ C_\alpha $ is a cofinal subset of $\alpha$ of order type $\cf(\alpha)$;
			\item\label{Definiton clubsuit - Clause guess} for every cofinal subset $ B\subseteq \kappa$, the set  $\{\alpha\in S \mid C_\alpha \subseteq B \}$ is stationary.
		\end{itemize}	
		\item \label{clubsuit^w(S)} $ \clubsuit^\mu_{J}(S,1) $ asserts the existence of a sequence $ \langle C_\alpha\mid \alpha\in S \rangle $ such that:
		\begin{itemize}
			\item for all $ \alpha\in S\cap \acc(\kappa) $, $ C_\alpha $ is a cofinal subset of $\alpha$ of order type $\cf(\alpha)$;
			\item\label{clubsuits_J_unboundedsubset} for every partition $\langle A_\beta \mid \beta<\mu \rangle$ of $\kappa$ there exists some $\beta<\mu$ such that the set $\{\alpha\in S \mid \sup(C_\alpha \cap A_\beta)=\alpha \}$ is stationary.
		\end{itemize} 
	\end{enumerate}
\end{definition}

Recall that by a Theorem of Shelah \cite{Sh_922}, for every uncountable cardinal $\lambda$ which satisfy $2^\lambda = \lambda^+$ and every stationary $S\subseteq E^{\lambda^+}_{\neq \cf(\lambda)}$, $\Diamond(S)$ holds.
It is clear that $\Diamond(S) \Rightarrow \clubsuit(S) \Rightarrow \clubsuit^\lambda_{J}(S,1) $.
The main corollary of this subsection is:

\begin{cor}\label{Corollary - directed set D_{mathcal C}}
	Let $\theta<\lambda$ be two regular cardinals.
	Assume $\lambda^{\theta}<\lambda^+$ and $\clubsuit^\lambda_{J}(S,1) $ holds for some stationary $S\subseteq E^{\lambda^+}_\theta$.
	Suppose $C$ and $E$ are two directed sets such that $\cf(C)<\lambda^+$ and $\non(\mathcal I_{\bd} (E))>\theta$ or $E\equiv_T 1$.
	Then there exists a directed set $D_{\mathcal C}$ such that:
	$$C\times [\lambda]^{<\theta}\times [\lambda^+]^{\leq \theta}\times E<_T C\times [\lambda]^{<\theta}\times D_{\mathcal C} \times E <_T C\times [\lambda^+]^{< \theta} \times E.$$
\end{cor}
In the rest of this subsection we prove this result.

Suppose $\mathcal C:=\langle C_\alpha \mid \alpha \in S \rangle$ is a $C$-sequence for some stationary set $S\subseteq E^{\lambda^+}_\theta$, i.e. $C_\alpha$ is a cofinal subset of $\alpha$ of order-type $\theta$, whenever $\alpha\in S$.
We define the directed set $D_{\mathcal C} :=\{ Y\in [\lambda^+]^{\leq \theta}\mid \forall \alpha\in S [|Y\cap C_\alpha|<\theta] \}$ ordered by $\subseteq$.
Notice that $\non(\mathcal I_{\bd}(D_{\mathcal C}))= \theta$ and $[\lambda^+]^{< \theta}\subseteq D_{\mathcal C}$.

Recall that by Hausdorff's formula $(\lambda^+)^\theta = \max\{\lambda^+,\lambda^\theta\}$, so if $\lambda^{\theta}<\lambda^+$, then $(\lambda^+)^\theta=\lambda^+$.
So we may assume $|D_{\mathcal C}|=\lambda^+$.

\begin{claim}\label{Claim - D_C from below}
	Suppose $|D_{\mathcal C}|=\lambda^+$, then $[\lambda^+]^{\leq \theta} \leq_T D_{\mathcal C}$.
\end{claim}
\begin{proof}
	Fix a bijection $\phi:D_{\mathcal C}\rightarrow \lambda^+$. Denote $X:=\{x\cup\{\phi(x)\}\mid x\in D_{\mathcal C}\}$, clearly $X$ is cofinal subset of $D_{\mathcal C}$.
	Let us fix some injective function $g:[\lambda^+]^{\leq \theta} \rightarrow X$.
	We claim that $g$ is a Tukey function, which witness that $[\lambda^+]^{\leq \theta} \leq_T D_{\mathcal C}$.
	Fix some $B\subseteq [\lambda^+]^{\leq \theta}$ unbounded in $ [\lambda^+]^{\leq \theta}$, note that $|B|>\theta$.
	As $g$ is injective, we get that $g"B$ is a set of size $>\theta$.
	Notice that there exists $Z\in [\lambda^+]^{\theta^+}$ such that $Z\subseteq \bigcup g"B$. 
	Assume that $g"B$ is bounded by $d\in D_{\mathcal C}$ in $D_{\mathcal C}$.
	As $D_{\mathcal C}$ is ordered by $\subseteq$, we get that $Z\subseteq d$, so $|d|\geq\theta^+$.
	But this is a absurd as every set in $D_{\mathcal C}$ is of size $\leq \theta$.
\end{proof}

Notice that by Lemma~\ref{Lemma - below kappa theta} and Claim~\ref{Claim - D_C from below}, as $(\lambda^+)^\theta=\lambda^+$ we have $ [\lambda]^{<\theta}\times [\lambda^+]^{\leq \theta} \leq _T [\lambda]^{<\theta} \times D_{\mathcal C} \leq _T [\lambda^+]^{< \theta}.$
Hence, $C\times [\lambda]^{<\theta}\times [\lambda^+]^{\leq \theta}\times E\leq _T C\times [\lambda]^{<\theta} \times D_{\mathcal C}\times E \leq _T C\times [\lambda^+]^{< \theta} \times E.$

\begin{claim} Suppose $\mathcal C$ is a $ \clubsuit^\lambda_{J}(S,1) $-sequence and:
	\begin{enumerate}[label=(\roman*)]

		\item $C$ is a directed set such that $|C|<\lambda^+$;	
		\item $E$ is a directed set such that $\non(\mathcal I_{\bd}(E)) >\theta$ and $\cf(E)\geq \lambda^+$.
	\end{enumerate}
	Then $ C\times D_{\mathcal C} \not\leq_T C\times E$.
\end{claim}
\begin{subproof}
Suppose that $f: C\times D_{\mathcal C} \rightarrow  C\times E$ is a Tukey function.
Fix some $o\in C$ and for each $\xi<\lambda^+$, denote $(c_\xi, x_\xi) := f(o,\{\xi\})$.
Consider the set $\{(c_\xi,x_\xi) \mid \xi<\lambda^+\}$.
For every $c\in C$, we define $A_c:=\{\xi <\lambda^+ \mid c_\xi =c\}$, clearly $\langle A_c \mid c\in C \rangle$ is a partition of $\lambda^+$ to less than $\lambda^+$ many sets.

As $\mathcal C$ is a $ \clubsuit^\lambda_{J}(S,1) $-sequence, there exists some $c\in C$ and $\alpha\in S$ such that $|C_\alpha\cap A_c|=\theta$.
Let us fix some $B\in [C_\alpha \cap A_c]^{\theta}$.
Notice that the set $G:=\{(o,\{\xi\})\mid \xi\in B\}$ is unbounded in $ C\times D_{\mathcal C}$, hence as $f$ is Tukey, $f"G$ is unbounded in $C\times E$.
The subset $\{x_\xi \mid \xi\in B\}$ of $E$ is of size $\theta$, hence bounded by some $e$.
Note that $f"G=\{( c,x_\xi)\mid \xi\in B\}$ is bounded by $(c,e)$ in $ C\times E$ which is absurd.
\end{subproof}
By the previous Claim, as $\lambda^\theta<\lambda^+$, we get that $ C\times D_{\mathcal C}\times [\lambda]^{<\theta} \times E \not \leq_T C\times [\lambda]^{<\theta}\times [\lambda^+]^{\leq \theta}\times E $.
The following Claim gives a negative answer to the question of whether there is a $C$-sequence $\mathcal C$ such that $D_{\mathcal C} \equiv_T [\lambda^+]^{<\theta}$.

In the following claim we use the fact that the sets in the sequence $\mathcal C$ are of a bounded cofinality.
\begin{claim}\label{gch imply no D_C equivalent to omega_2 finite}
	Assume $\lambda^\theta<\lambda^+$.
	Suppose $S\subseteq E^{\lambda^+}_\theta$ is a stationary set and $\mathcal C:=\langle C_\alpha  \mid \alpha\in S \rangle$ is a $C$-sequence, then $D_{\mathcal C}\not \geq_T [\lambda^+]^{<\theta}$.
\end{claim}
\begin{subproof}
	Let $S\subseteq E^{\lambda^+}_\theta$ and $\mathcal C:=\langle C_\alpha  \mid \alpha\in S \rangle$ be a $C$-sequence.
	Suppose we have $ [\lambda^+]^{< \theta} \leq_T D_{\mathcal C} $, let $f:[\lambda^+]^{< \theta} \rightarrow D_{\mathcal C}$ be a Tukey function and $Y:=f"[\lambda^+]^{1}$.
	Let us split to two cases:
	
	$\br$ Suppose $|Y|<\lambda^+$. By the pigeonhole principle, we can find a subset $Q\subseteq [\lambda]^1$ of size $\theta$ such that $f"Q=\{x\}$ for some $x\in D_{\mathcal C}$.
	As $f$ is Tukey and $Q$ is unbounded in $[\lambda^+]^{< \theta}$, the set $f"Q$ is unbounded which is absurd.
	
	$\br$ Suppose $|Y|=\lambda^+$.
	As $f$ is Tukey, every subset of $Y$ of size $\theta$ is unbounded which is absurd to the following claim.
	\begin{subclaim}\label{good claim}
		There is no subset $Y\subseteq D_{\mathcal C}$ of size $\lambda^+$ such that every subset of $Y$ of size $\theta$ is unbounded.
	\end{subclaim}
	\begin{subproof}
		Assume towards a contradiction that $Y$ is such a set.
		As $\lambda^{\theta}<\lambda^+$, we may refine $Y$ and assume that $Y=\{y_\alpha \mid \alpha<\lambda^+\}$ is a $\Delta$-system with a root $R$ separated by a club $C\subseteq \lambda^+$, i.e. such that for every $\alpha<\beta<\lambda^+$, $y_\alpha\setminus R <\eta < y_\beta \setminus  R$ for some $\eta\in C$.
		
		We define an increasing sequence of ordinals $\langle \beta_\nu\mid \nu\leq\theta^2\rangle$ where for each $\nu\leq\theta^2$ we let $\beta_\nu:=\sup\{y_\xi \mid \xi<\nu\}$.
		As $C$ is a club, we get that $\beta_{\theta\cdot \nu}\in C$ for each $\nu <\theta$.
		
		We aim to construct a subset $X=\{x_j \mid j <\theta\}$ of $Y$, we split to two cases:
		Suppose $\beta_{\theta^2}\in S$.
		Recall that $\otp (C_{\beta_{\theta^2}})=\theta$ and $\sup(C_{\beta_{\theta^2}})=\beta_{\theta^2}$, so for every $j <\theta$ we have that the interval $[\beta_{\theta\cdot j},\beta_{\theta\cdot (j+1)})$ contains $<\theta$ many elements of the ladder $C_{\beta_{\theta^2}}$, let us fix some $x_j\in Y$ such that $x_j\setminus R\subset [\beta_{\theta\cdot j},\beta_{\theta\cdot (j+1)})$ and $x_j\setminus R$ is disjoint from $C_{\beta_{\theta^2}}$.
		If $\beta_{\theta^2}\notin S$, define $X:=\{x_j \mid j<\theta\}$ where $x_j:= y_{\theta\cdot j}$.
		
		Let us show that $X=\{x_j \mid j <\theta\}$ is a bounded subset of $Y$, which is a contradiction to the assumption.
		It is enough to show that for every $\alpha\in S$, we have that $|(\bigcup X)\cap C_\alpha|<\theta$.
		Let $\alpha\in S$.
		
		$\br$ Suppose $\alpha>\beta_{\theta^2}$, as $C_\alpha$ is a cofinal subset of $\alpha$ of order-type $\theta$ and $\bigcup X$ is bounded by $\beta_{\theta^2}$ it is clear that $|(\bigcup X)\cap C_\alpha|<\theta$.
		
		$\br$ Suppose $\alpha<\beta_{\theta^2}$. As $C_\alpha$ if cofinal in $\alpha$ and of order-type $\theta$, there exists some $j <\theta$ such that for all $j <\rho<\theta$, we have $(x_\rho\setminus R)\cap C_\alpha =\emptyset$.
		As $x_\rho\in D_{\mathcal C_R}$ for every $\rho<\theta$ and $\theta$ is regular, we get that $|(\bigcup X)\cap C_\alpha |<\theta$ as sought.
		
		$\br$ Suppose $\alpha = \beta_{\theta^2}$.
		Notice this implies that we are in the first case of the construction of the set $X$.
		Recall that the $\Delta$-system $\{x_j \mid j <\theta\}$ is such that $(x_j\setminus R )\cap C_\alpha =\emptyset$, hence $(\bigcup X)\cap C_{\alpha} = R \cap C_\alpha $.
		Recall that as $x_0\in D_{\mathcal C}$, we get that $R \cap C_\alpha$ is of size $<\theta$, hence also $(\bigcup X)\cap C_{\alpha}$ is as sought.
	\end{subproof}\qedhere
\end{subproof}

\begin{claim}Assume $\lambda^\theta<\lambda^+$.
	Suppose $C$ and $E$ are two directed sets such that $|C|<\lambda^+$ and either $\non(\mathcal I_{\bd} (E))>\theta$ or $E\equiv_T 1$.
	Then for every $C$-sequence $\mathcal C$ on a stationary $S\subseteq E^{\lambda^+}_\theta$, $C\times [\lambda^+]^{<\theta} \times E\not\leq_T C\times  D_{\mathcal C} \times E$.
\end{claim}
\begin{subproof}
	Let $\mathcal C:=\langle C_\alpha \mid \alpha\in S \rangle $ be a $C$-sequence where $S\subseteq E^{\lambda^+}_\theta$.
	Suppose on the contrary that $C\times [\lambda^+]^{<\theta} \times E \leq_T C\times  D_{\mathcal C} \times E$.
	Hence, $[\lambda^+]^{<\theta}\leq_T C\times  D_{\mathcal C} \times E$, let us fix a Tukey function $f:[\lambda^+]^{<\theta}\rightarrow C\times  D_{\mathcal C}\times E$ witnessing that.
	Consider $X=[\lambda^+]^1$.
	
	By the pigeonhole principle, there exists some $c\in  C$ and some set $Z\subseteq X$ of size $\lambda^+$ such that $f"Z \subseteq  \{ c\}\times D_{\mathcal C}\times E$.
	Let $Y:=\pi_{D_{\mathcal C}} (f"Z)$.
	Let us split to two cases:

	$\br$ Suppose $|Y|< \lambda^+$. By the pigeonhole principle, we can find a subset $Q\subseteq Z$ of size $\theta$ such that $f"Q=\{c\}\times \{x\}\times E$ for some $x\in D_{\mathcal C}$.
	As $f$ is Tukey and $Q$ is unbounded, we must have that $f"Q$ is unbounded, but this is absurd as $\non(\mathcal I_{\bd} (E))>\theta$.
	
	$\br$ Suppose $|Y|=\lambda^+$.
	As $f$ is Tukey and either $\non(\mathcal I_{\bd} (E))>\theta$ or $E=1$, every subset of $Y$ of size $\theta$ is unbounded which is impossible by Claim~\ref{good claim}.
\end{subproof}

\subsection{Structure of $D_{\mathcal C}$}

In \cite[Lemmas~1,2]{MR792822}, Todor\v{c}evi\'{c} defined for every $\kappa$ regular and $S\subseteq \kappa$ the direced set $D(S):=\{C\subseteq [S]^{\leq\omega} \mid \forall \alpha<\omega_1[\sup(C\cap \alpha)\in C]\}$ ordered by inclusion; and studied the structure of such directed sets.
In this section we follow this line of study but for directed sets of the form $D_{\mathcal C}$, constructing a large $<_T$-antichain and chain of directed sets using $\theta$-support product.

\subsubsection{Antichain}

\begin{theorem}\label{antichain of D_C}
	Suppose $2^\lambda=\lambda^+$, $\lambda^\theta<\lambda^+$, then there exists a family $\mathcal F$ of size $2^{\lambda^+}$ of directed sets of the form $D_{\mathcal C}$ such that every two of them are Tukey incomparable.
\end{theorem}
\begin{proof}
	As $2^\lambda=\lambda^+$ holds, by Shelah's Theorem we get that $\diamondsuit(S)$ holds for every $S\subseteq E^{\lambda^+}_\theta$ stationary subset.
	Let us fix some stationary subset $S\subseteq E^{\lambda^+}_\theta$ and a partition of $S$ into $\lambda^+$-many stationary subsets $\langle S_\alpha \mid \alpha<\lambda^+ \rangle$.
	For each $S_\alpha$ we fix a $\clubsuit(S_\alpha)$ sequence $\langle C_\beta \mid \beta\in S_\alpha \rangle$.
	
	Let us fix a family $\mathcal F$ of size $2^{\lambda^+}$ of subsets of $S $ such that for every two $R,T\in \mathcal F$ there exists some $S_\alpha$ such that $R\setminus T \supseteq S_\alpha$.
	For each $T\in \mathcal F$ let us define a $C$-sequence $\mathcal C_T:= \langle C_\alpha \mid \alpha\in T \rangle$.
	Clearly the following Lemma shows the family $\{ D_{\mathcal C_T} \mid T\in \mathcal F \}$ is as sought.
	
	\begin{claim}\label{Usefull lemma}
		Suppose $\mathcal C_T:=\langle C_\beta \mid \beta \in T \rangle$ and $\mathcal C_R:=\langle C_\beta \mid \beta \in R \rangle$ are two $C$-sequences such that $T,R\subseteq E^{\lambda^+}_\theta$ are stationary subsets.
		Then if $\langle C_\beta \mid \beta \in T\setminus R \rangle$ is a $\clubsuit$-sequence, then $D_{\mathcal C_T}\not \leq_T D_{\mathcal C_R}$. 
	\end{claim}
	\begin{subproof}
		Suppose $f:D_{\mathcal C_T} \rightarrow D_{\mathcal C_R}$ is a Tukey function. 
		Fix a subset $W\subseteq [\lambda^+]^1\subseteq D_{\mathcal C_T}$ of size $\lambda^+$, we split to two cases.
		
		$\br$ Suppose $f"W \subseteq [\alpha]^{\theta}$ for some $\alpha<\lambda^+$.
		As $\lambda^\theta<\lambda^+$, by the pigeonhole principle we can find a subset $X\subseteq W $ of size $\lambda^+$ such that $f"X = \{z\}$ for some $z\in D_{\mathcal C_R}$.
		As $\langle C_\beta \mid \beta \in T\setminus R \rangle$ is a $\clubsuit$-sequence and $\bigcup X \in [\lambda^+]^{\lambda^+}$, there exists some $\beta \in T\setminus R$ such that $C_\beta \subseteq \bigcup X$. 
		So $X$ is an unbounded subset of $\mathcal C_T$ such that $f"X$ is bounded in $\mathcal C_R$ which is absurd.
		
		$\br$ As $|f"W|=\lambda^+$, using $\lambda^\theta<\lambda^+$ we may fix a subset $Y=\{y_\beta \mid \beta <\lambda^+\}\subseteq f"W$ which forms a $\Delta$-system with a root $R_1$.
		In other words, for $\alpha<\beta<\lambda^+$ we have $y_\alpha \setminus R_1 < y_\beta\setminus  R_1 $ and $y_\alpha\cap y_\beta =R_1$.
		For each $\alpha<\lambda^+$, we fix $x_\alpha \in W$ such that $f(x_\alpha)=y_\alpha$.
		Finally, without loss of generality we may use the $\Delta$-system Lemma again and refine our set $Y$ to get that there exists a club $E\subseteq \lambda^+$ such that, for all $\alpha<\beta<\lambda^+$ we have:
		\begin{itemize}
			\item $x_\alpha \cap x_\beta = \emptyset$;
			\item $y_\alpha\cap y_\beta = R_1 $;
			\item there exists some $\gamma \in E$ such that $x_\alpha < \gamma < x_\beta  $ and 
			$y_\alpha\setminus R_1 < \gamma < y_\beta \setminus R_1 $;
			\item $f(x_\alpha) = y_\alpha$.
		\end{itemize}
		Furthermore we may assume that between any two elements of $\xi<\eta$ in $E$ there exists a unique $\alpha<\lambda^+$ such that $\xi< x_\alpha \cup (y_\alpha\setminus R_1)<\eta$.
		
		As $\langle C_\beta \mid \beta \in T\setminus R \rangle$ is a $\clubsuit$-sequence, there exists some $\beta \in (T\setminus R)\cap \acc(E)$ such that $C_\beta \subseteq \bigcup \{x_\alpha \mid \alpha<\lambda^+\}$. 
		Construct by recursion an increasing sequence $\langle \beta_\nu \mid \nu<\theta \rangle\subseteq C_\beta$ and a sequence $\langle z_\nu \mid \nu<\theta \rangle \subseteq \{x_\alpha \mid \alpha<\lambda^+\}$ such that $\beta_\nu\in z_\nu<\beta$.
		
		Clearly, $\{z_\nu\mid \nu<\theta\}$ is unbounded in $ D_{\mathcal C_T}$, so the following Claim proves $f$ is not a Tukey function.
		\begin{subclaim}
			The subset $\{ f(z_\nu)\mid \nu<\theta \}$ is bounded in $D_{\mathcal C_R}$.
		\end{subclaim}
		\begin{subproof}
			Let $Y:=\bigcup f(z_\nu)$ and $\mathcal C_R:=\langle C_\beta \mid \beta \in R \rangle$, we will show that for every $\alpha\in R$, we have $|Y\cap C_\alpha |<\theta$.
			By the refinement we did previously it is clear that $\{f(z_\nu)\setminus R_1\mid \nu<\theta\}$ is a pairwise disjoint sequence, where for each $\nu<\theta$ we have some element $\gamma_\nu \in E$ such that $f(z_\nu)\setminus R_1 < \gamma_\nu < f(z_{\nu+1}) \setminus R_1<\beta $.
			Let $\alpha\in R$.
			
			$\br$ Suppose $\alpha>\beta$. As $C_\alpha$ is cofinal in $\alpha$ and of order-type $\theta$, then $|Y\cap C_\alpha |<\theta$.
			
			$\br$ Suppose $\alpha<\beta$. As $C_\alpha$ is cofinal in $\alpha$ and of order-type $\theta$, there exists some $\nu<\theta$ such that for all $\nu<\rho<\theta$, we have $(f(z_\rho)\setminus R_1)\cap C_\alpha =\emptyset$.
			As $f(z_\rho)\in D_{\mathcal C_R}$ for every $\rho<\theta$ and $\theta$ is regular, we get that $|Y\cap C_\alpha |<\theta$ as sought.
			
			As $\beta \notin R$ there are no more cases to consider.\qedhere
		\end{subproof}\qedhere
	\end{subproof}\qedhere
\end{proof}

\begin{cor}
	Suppose $2^\lambda=\lambda^+$, $\lambda^\theta<\lambda^+$ and $S\subseteq E^{\lambda^+}_\theta$ is a stationary subset.
	Then there exists a family $\mathcal F$ of directed sets of the form $D_{\mathcal C}\times[\lambda]^{<\theta}$ of size $2^{\lambda^+}$ such that every two of them are Tukey incomparable.
\end{cor}
\begin{proof}
	Clearly by the same arguments of Theorem~\ref{antichain of D_C} the following Lemma is suffices to get the wanted result.
	\begin{claim}
		Suppose $\mathcal C_T:=\langle C_\beta \mid \beta \in T \rangle$ and $\mathcal C_R:=\langle C_\beta \mid \beta \in R \rangle$ are two $C$-sequences such that $T,R\subseteq E^{\lambda^+}_\theta$ are stationary subsets such that $T\setminus R$ is stationary.
		Then if $\langle C_\beta \mid \beta \in T\setminus R \rangle$ is a $\clubsuit$-sequence, then $ D_{\mathcal C_T}\times [\lambda]^{<\theta }\not \leq_T D_{\mathcal C_R}\times [\lambda]^{<\theta }$. 
	\end{claim}
	\begin{subproof}
		Suppose $f: D_{\mathcal C_T} \times [\lambda]^{<\theta }\rightarrow  D_{\mathcal C_R}\times [\lambda]^{<\theta }$ is a Tukey function. 
		Consider $Q=f"([\lambda^+]^1\times \{\emptyset\} )$, let us split to two cases:
		
		$\br$ If $|Q|<\lambda^+$, then by the pigeonhole principle, there exists $x\in D_{\mathcal C_R}$, $F\in [\lambda]^{<\theta }$ and a set $W\subseteq [\lambda^+]^1$ of size $\lambda^+$ such that $f"( W\times \{\emptyset\}) =\{(x,F)\}$.
		As $\langle C_\beta \mid \beta \in T\setminus R \rangle$ is a $\clubsuit$-sequence and $\bigcup W \in [\lambda^+]^{\lambda^+}$, we may fix some $\beta\in T\setminus R$ such that $C_\beta \subseteq \bigcup  W$.
		Hence $W\times \{\emptyset\}$ is unbounded in $D_{\mathcal C_T}\times [\lambda]^{<\theta }$ but $f"(W\times \{\emptyset\})$ is bounded in $D_{\mathcal C_R}\times [\lambda]^{<\theta }$ which is absurd as $f$ is Tukey.
		
		$\br$ If $|Q|=\lambda^+$, then by the pigeonhole principle there exists some $F\in [\lambda]^{<\theta }$ and a set $W\subseteq [\lambda^+]^1$ of size $\lambda^+$ such that, $f"(W\times \{\emptyset\}) \subseteq D_{\mathcal C_R}\times\{F\}   $.
		Let $Y:=\pi_0(f"(W\times \{\emptyset\}))$.
		Next we may continue with the same proof as in Lemma~\ref{Usefull lemma}.
	\end{subproof}
\end{proof}

\subsubsection{Chain}

\begin{theorem}\label{Theorem - D_C chain}
	Suppose $2^\lambda=\lambda^+$, $\lambda^\theta<\lambda^+$.
	Then there exists a family $\mathcal F=\{D_{\mathcal C_\xi} \mid \xi<\lambda^+\}$ of Tukey incomparable directed sets of the form $D_{\mathcal C}$ such that $\langle \prod^{\leq\theta}_{\zeta <\xi} D_{\mathcal C_\zeta} \mid \xi<\lambda^+\rangle$ is a $<_T$-increasing chain.
\end{theorem}
\begin{proof}
As in Theorem~\ref{antichain of D_C}, we fix a partition $\langle S_\zeta \mid \zeta<\lambda^+\rangle$ of $E^{\lambda^+}_\theta$ to stationary subsets such that there exists a $\clubsuit(S_\zeta) $-sequence $\mathcal C_\zeta$ for $\zeta<\lambda^+$.
Note that for every $A\in [\lambda^+]^{<\lambda^+}$, we have $|\prod^{\leq \theta}_{\zeta\in A} D_{\mathcal C_\zeta}|=\lambda^+$.
Note that for every $A,B \in [\lambda^+]^{<\lambda^+}$ such that $A\subset B$, we have $\prod^{\leq \theta}_{\zeta\in A} D_{\mathcal C_\zeta} \leq_T \prod^{\leq \theta}_{\zeta\in B} D_{\mathcal C_\zeta}$.
The following Claim gives us the wanted result.

\begin{claim}
	Suppose $A\in [\lambda^+]^{<\lambda^+}$ and $\xi\in \lambda^+\setminus A$, then $ D_{\mathcal C_\xi} \not \leq_T \prod^{\leq \theta}_{\zeta\in A} D_{\mathcal C_\zeta} $.
	In particular, $\prod^{\leq \theta}_{\zeta\in A} D_{\mathcal C_\zeta} <_T \prod^{\leq \theta}_{\zeta\in A} D_{\mathcal C_\zeta}\times  D_{\mathcal C_\xi}$.
\end{claim}

\begin{subproof}
	Let $D:=D_{\mathcal C_\xi}$ and $E:=\prod^{\leq \theta}_{\zeta\in A} D_{\mathcal C_\zeta}$.
	Note that as $2^\lambda=\lambda^+$, then $(\lambda^+)^\lambda= \lambda^+$, so $|E|=\lambda^+$.
	Suppose $f:D \rightarrow  E$ is a Tukey function.
	Consider $Q=f"[\lambda^+]^1$, let us split to cases:
	
	$\br$ Suppose $|Q|< \lambda^+$, then by pigeonhole principle, there exists $e\in E$ and a subset $X\subseteq D$ of size $\lambda^+$ such that $f"X =\{e\}$.
	As $\langle C_\beta \mid \beta \in S_\xi \rangle$ is a $\clubsuit$-sequence, there exists some $\beta \in S_\xi$ such that $C_\beta \subseteq \bigcup X$. 
	So $X$ is an unbounded subset of $D$ such that $f"X$ is bounded in $E$ which is absurd.

	$\br$ Suppose $|Q|=\lambda^+$.
	Let us enumerate $Q:=\{q_\alpha \mid \alpha<\lambda^+\}$.
	Recall that for every $\zeta\in A$, $D_{\mathcal C_{\zeta}}\subseteq [\lambda^+]^{\leq \theta}$.
	Let $z_\alpha:= \bigcup \{ q_\alpha(\zeta) \times \{\zeta\} \mid \zeta\in A, ~q_\alpha(\zeta)\neq 0_{D_{\mathcal C_\zeta}}\}$, notice that $z_\alpha \in [\lambda^+ \times A]^{\leq\theta}$.
	We fix a bijection $\phi:\lambda^+\times A \rightarrow \lambda^+$.
	
	As $\{\phi"z_\alpha\mid \alpha<\lambda^+\}$ is a subset of $[\lambda^+]^{\leq \theta}$ of size $\lambda^+$ and $\lambda^\theta<\lambda^+$, by the $\Delta$-system Lemma, we may refine our sequence $Q$ and re-index such that $\{\phi "z_\alpha \mid \alpha<\lambda^+\}$ will be a $\Delta$-system with root $R'$.
	
	For each $\alpha<\lambda^+$ and $\zeta\in A$, let $y_{\alpha,\zeta}:=\{ \beta<\lambda^+ \mid \beta\in q_\alpha(\zeta)\}$.
	We claim that for each $\zeta\in A$, the set $\{y_{\alpha,\zeta}\mid \zeta \in A\}$ is a $\Delta$-system with root $R_{\zeta}:=\{\beta<\lambda^+\mid (\beta,\zeta)\in \phi^{-1}[R'] \}$.
	Let us show that whenever $\alpha<\beta<\lambda^+$, we have $y_{\alpha,\zeta}\cap y_{\beta,\zeta} = R_{\zeta} $.
	Notice that $\delta\in y_{\alpha,\zeta}\cap y_{\beta,\zeta} \iff \delta\in q_\alpha(\zeta)\cap q_\beta(\zeta) \iff (\delta,\zeta)\in z_\alpha\cap z_\beta \iff \phi(\delta,\zeta) \in \phi"(z_\alpha\cap z_\beta)=\phi"z_\alpha\cap \phi"z_\beta = R' \iff (\delta,\zeta)\in \phi^{-1}R' \iff \delta\in R_{\zeta}$.
	For each $\alpha<\lambda^+$, we fix $x_\alpha \in [\lambda^+]^1$ such that $f(x_\alpha)=q_\alpha$.
	
	We use the $\Delta$-system Lemma again and refine our sequence such that there exists a club $C\subseteq \lambda^+$ and for all $\alpha<\beta<\lambda^+$ we have:
	\begin{enumerate}
		\item for every $\zeta\in A$, we have $y_{\alpha,\zeta}\cap y_{\beta,\zeta} = R_{\zeta} $;
		\item $x_\alpha \cap x_\beta= \emptyset$;
		\item there exists some $\gamma \in C$ such that $x_\alpha \cup(\bigcup_{\zeta\in A} (y_{\alpha,\zeta}\setminus R_\zeta ))< \gamma < x_\beta \cup(\bigcup_{\zeta\in A } (y_{\beta,\zeta}\setminus R_\zeta) )$.
	\end{enumerate} 
	Furthermore, we may assume that between any two elements of $\gamma<\delta$ in $C$ there exists some $\alpha<\lambda^+$ such that $\gamma<x_\alpha\cup(\bigcup_{\zeta\in A } (y_{\alpha,\zeta}\setminus R_\zeta ))<\delta$.
	We continue in the spirit of Claim~\ref{good claim}.
	
	As $\langle C_\beta \mid \beta \in S_\xi \rangle$ is a $\clubsuit$-sequence, there exists some $\beta \in S_\xi\cap \acc(C)$ such that $C_\beta \subseteq \bigcup \{x_\alpha \mid \alpha<\lambda^+\}$. 
	Construct by recursion an increasing sequence $\langle \beta_\nu\mid \nu<\theta \rangle\subseteq C_\beta$ and a sequence $\langle w_\nu \mid \nu<\theta \rangle \subseteq \{x_\alpha \mid \alpha<\lambda^+\}$ such that $\beta_\nu\in w_\nu<\beta$.
	
	Clearly, $\{w_\nu\mid \nu<\theta\}$ is unbounded in $ D_{\mathcal C_\xi}$, so the following Claim proves $f$ is not a Tukey function.
	\begin{subclaim}
		The subset $\{ f(w_\nu)\mid \nu<\theta \}$ is bounded in $E$.
	\end{subclaim}
	\begin{subproof}
		For each $\zeta\in A$, let $W_\zeta:= \bigcup_{\nu<\theta} f(w_\nu)(\zeta)$.
		We will show that $W_\zeta \in D_{\mathcal C_\zeta}$, as  $|\{\zeta\in A\mid W_\zeta\neq \emptyset\}|\leq \theta$ this will imply that $\prod^{\leq \theta}_{\zeta \in A} W_\zeta$ is well defined and an element of $E$.
		Clearly $f(w_\nu)\leq_E \prod^{\leq \theta}_{\zeta \in A} W_\zeta$ for every $\nu<\theta$, so the set $\{ f(w_\nu)\mid \nu<\theta \}$ is bounded in $E$ as sought.
		
		Let $\mathcal C_{\zeta}:=\langle C_\alpha \mid \alpha \in S_\zeta \rangle$, we will show that for every $\alpha\in S_\zeta$, we have $|W_\zeta \cap C_\alpha |<\theta$.
		By the refinement we did previously it is clear that $\{f(w_\nu)(\zeta)\setminus R_\zeta\mid \nu<\theta\}$ is a pairwise disjoint sequence, where for each $\nu<\theta$ we have some element $\gamma_\nu \in C$ such that $f(w_\nu)(\zeta)\setminus R_\zeta < \gamma_\nu < f(w_{\nu+1})(\zeta)\setminus R_\zeta $.
		Furthermore, $f(w_\nu)(\zeta)\subseteq \beta$ for every $\nu<\theta$.
		Let $\alpha\in S_\zeta$.
		
		$\br$ Suppose $\alpha>\beta$. As $C_\alpha$ if cofinal in $\alpha$ and of order-type $\theta$, then $|W_\zeta\cap C_\alpha |<\theta$.
		
		$\br$ Suppose $\alpha<\beta$. As $C_\alpha$ if cofinal in $\alpha$ and of order-type $\theta$, there exists some $\nu<\theta$ such that for all $\nu<\rho<\theta$, we have $(f(w_\rho)(\zeta)\setminus R_\zeta )\cap C_\alpha =\emptyset$.
		As $f(w_\rho)(\zeta)\in D_{\mathcal C_\zeta}$ for every $\rho<\theta$, we get that $|W_\zeta\cap C_\alpha |<\theta$ as sought.
		
		As $\beta \notin S_\zeta$ there are no more cases to consider.\qedhere
	\end{subproof}\qedhere
\end{subproof}
\end{proof}

\section{Concluding remarks}\label{Section - 6}

A natural continuation of this line of research is analysing the class $\mathcal D_{\kappa}$ for cardinals $\kappa\geq \aleph_\omega$. 
As a preliminary finding we notice that the poset $(\mathcal P(\omega),\subset)$ can be embdedded by a function $\mathfrak F$ into the class $\mathcal D_{\aleph_\omega}$ under the Tukey order.
Furthermore, for every two successive elements $A,B$ in the poset $(\mathcal P(\omega),\subset)$, i.e. $A\subset B $ and $|B\setminus A|=1$, there is no directed set $D$ such that $\mathfrak F(A)<_T D <_T \mathfrak F(B)$.
The embedding is defined via $\mathfrak F(A):= \prod^{<\omega}_{n\in A} \omega_{n+1}$, and the furthermore part can be proved by Lemma~\ref{Lemma - 2nd gap}.
As a corollary, we get that in $\zfc$ the cardinality of $\mathcal D_{\aleph_\omega}$ is at least $2^{\aleph_0}$.

\section*{Acknowledgments}
This paper presents several results from the author’s PhD research at Bar-Ilan University under the
supervision of Assaf Rinot to whom he wishes to express his deep appreciation.
The author is supported by the European Research Council (grant agreement ERC-2018-StG 802756).
	
Our thanks go to Tanmay Inamdar for many illuminating discussions.
We thank the referee for their effort and for writing a detailed thoughtful report that improved this paper.

\newpage
\renewcommand{\thesection}{A}
\section{Appendix: Tukey ordering of simple elements of the class $\mathcal D_{\aleph_2}$ and $\mathcal D_{\aleph_3}$}

We present each of the posets $(\mathcal T_{2},<_T)$ and $(\mathcal T_{3},<_T)$ in a diagram.
In both diagrams below, for any two directed sets $A$ and $B$, an arrow $A\rightarrow B$, represents the fact that $A<_T B$.
If the arrow is dashed, then under $\gch$ there exists a directed set in between.
If the arrow is not dashed, then there is no directed set in between $A$ and $B$.
Every two directed sets $A$ and $B$ such that there is no directed path (in the obvious sense) from $A$ to $B$, are such that $A\not\leq_T B$.
Note that this implies that any two directed sets on the same horizontal level are incompatible in the Tukey order.

\begin{figure}[h!]\label{Figure T_2}
	\centering
	\begin{tikzcd}
		& {[\omega_2]^{<\omega}}                                          &                                                    \\
		& {[\omega_2]^{\leq\omega} \times [\omega_1]^{<\omega}} \arrow[dashed, u] &                                                    \\
		{\omega_2 \times [\omega_1]^{<\omega}} \arrow[dashed, ru] &                                                                 & {\omega \times [\omega_2]^{\leq\omega}} \arrow[dashed, lu] \\
		{[\omega_1]^{<\omega}} \arrow[black, u]                  & \omega \times \omega_1 \times \omega_2 \arrow[dashed,ru] \arrow[dashed, lu]    & {[\omega_2]^{\leq\omega}} \arrow[black, u]                \\
		\omega \times \omega_1 \arrow[dashed, u] \arrow[black, ru]       & \omega \times \omega_2 \arrow[black, u]                                & \omega_1 \times \omega_2 \arrow[dashed, u] \arrow[black, lu]      \\
		\omega \arrow[black, u] \arrow[black, ru]                       & \omega_1 \arrow[black, lu] \arrow[black, ru]                                  & \omega_2 \arrow[black, lu] \arrow[black, u]                      \\
		& 1 \arrow[black, lu] \arrow[black, u] \arrow[black, ru]                               &                                                   
		\\
	\end{tikzcd}
	\caption{Tukey ordering of $(\mathcal T_{2},<_T)$}
\end{figure}
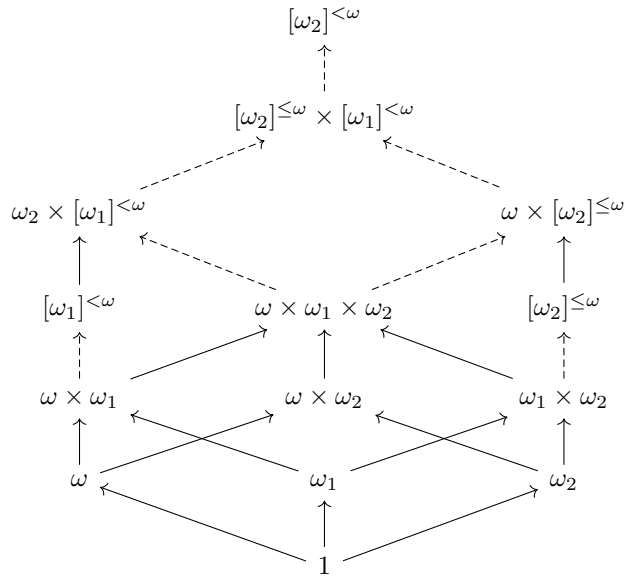
\newpage

\begin{landscape}
	\begin{figure}[!h]\label{Figure T_3}
		\centering
		\begin{tikzcd}[column sep=0, row sep=20]
			&                                                                                     &                                                                                     &                                                                                     &                                {[\omega_3]^{<\omega}}                                &                                                                                     &                                                                                     &                                                                                      \\
			&                                                                                     &                                                                                     &                                                                                     &                {[\omega_3]^{\leq\omega} \times [\omega_2]^{<\omega}}                 \arrow[dashed, u]&                                                                                     &                                                                                     &                                                                                      \\
			&                                                                                     &                                                                                     &               {[\omega_3]^{\leq\omega_1} \times [\omega_2]^{<\omega}}                \arrow[dashed, ru]&                                                                                     &                {[\omega_3]^{\leq\omega} \times [\omega_1]^{<\omega}}                 \arrow[dashed, lu]&                                                                                     &                                                                                      \\
			&                                                                                     &                                                                                     &                        {\omega_3 \times [\omega_2]^{<\omega}}                        \arrow[dashed, u]& {[\omega_3]^{\leq\omega_1} \times [\omega_2]^{\leq\omega} \times [\omega_1]^{<\omega}} \arrow[dashed, lu]\arrow[dashed, ru]&                       {\omega \times [\omega_3]^{\leq\omega}}                        \arrow[dashed, u]&                                                                                     &                                                                                      \\
			&                                                                                     &                                {[\omega_2]^{<\omega}}                                \arrow[black, ru]&        {\omega_3 \times [\omega_2]^{\leq\omega} \times [\omega_1]^{<\omega}}         \arrow[dashed, u]\arrow[dashed, ru]&               {[\omega_3]^{\leq\omega_1} \times [\omega_1]^{<\omega}}                \arrow[dashed, u]&       {\omega \times [\omega_3]^{\leq\omega_1} \times [\omega_2]^{\leq\omega}}       \arrow[dashed, u]\arrow[dashed, lu]&                              {[\omega_3]^{\leq\omega}}                               \arrow[dashed, black, lu]&                                                                                      \\
			&                                                                                     &                {[\omega_2]^{\leq\omega} \times [\omega_1]^{<\omega}}                 \arrow[dashed, u]\arrow[black, ru]&                {\omega_2 \times \omega_3 \times [\omega_1]^{<\omega}}                \arrow[dashed, ru]\arrow[dashed, u]&               {\omega \times \omega_3 \times [\omega_2]^{\leq\omega}}                \arrow[dashed, ru]\arrow[dashed, lu]&              {\omega \times \omega_1 \times [\omega_3]^{\leq\omega_1}}               \arrow[dashed, lu]\arrow[dashed, u]&              {[\omega_3]^{\leq\omega_1} \times [\omega_2]^{\leq\omega}}              \arrow[dashed, u]\arrow[black, lu]&                                                                                      \\
			&                        {\omega_2 \times [\omega_1]^{<\omega}}                        \arrow[dashed, ru]\arrow[black, rru]&                       {\omega \times [\omega_2]^{\leq\omega}}                        \arrow[dashed, u]\arrow[black, rru]&                        {\omega_3 \times [\omega_1]^{<\omega}}                        \arrow[black, u]&               {\omega \times \omega_1 \times \omega_2 \times \omega_3}               \arrow[dashed, ru]\arrow[dashed, u]\arrow[dashed, lu]&                      {\omega \times [\omega_3]^{\leq\omega_1}}                       \arrow[black, u]&                      {\omega_3 \times [\omega_2]^{\leq\omega}}                       \arrow[dashed, u]\arrow[black, llu]&                     {\omega_1 \times [\omega_3]^{\leq\omega_1}}                      \arrow[dashed, lu]\arrow[black, llu] \\
			&                                {[\omega_1]^{<\omega}}                                \arrow[black, u]\arrow[black, rru]&                       {\omega \times \omega_1 \times \omega_2}                       \arrow[dashed, u]\arrow[dashed, lu]\arrow[black, rru]&                       {\omega \times \omega_1 \times \omega_3}                       \arrow[dashed, u]\arrow[black, ru]&                              {[\omega_2]^{\leq\omega}}                               \arrow[black, llu]\arrow[black, rru]&                       {\omega \times \omega_2 \times \omega_3}                       \arrow[dashed, u]\arrow[black, lu]&                      {\omega_1 \times \omega_2 \times \omega_3}                      \arrow[dashed, ru]\arrow[dashed, u]\arrow[black, llu]&                             {[\omega_3]^{\leq\omega_1}}                              \arrow[black, llu]\arrow[black, u] \\
			&                               {\omega \times \omega_1}                               \arrow[dashed, u]\arrow[black, ru]\arrow[black, rru]&                               {\omega \times \omega_2}                               \arrow[black, u]\arrow[black, rrru]&                               {\omega \times \omega_3}                               \arrow[black, u]\arrow[black, rru]&                              {\omega_1 \times \omega_2}                              \arrow[dashed, u]\arrow[black, llu]\arrow[black, rru]&                              {\omega_1 \times \omega_3}                              \arrow[black, llu]\arrow[black, ru]&                              {\omega_2 \times \omega_3}                              \arrow[dashed, ru]\arrow[black, lu]\arrow[black, u]&                                                                                      \\
			&                                                                                     &                                       {\omega}                                       \arrow[black, lu]\arrow[black, u]\arrow[black, ru]&                                      {\omega_1}                                      \arrow[black, llu]\arrow[black, ru]\arrow[black, rru]&                                      {\omega_2}                                      \arrow[black, llu]\arrow[black, u]\arrow[black, rru]&                                      {\omega_3}                                      \arrow[black, llu]\arrow[black, u]\arrow[black, ru]&                                                                                     &                                                                                      \\
			&                                                                                     &                                                                                     &                                                                                     &                                         {1}                                          \arrow[black, llu]\arrow[black, lu]\arrow[black, u]\arrow[black, ru]&                                                                                     &                                                                                     &                                                                                      \\
		\end{tikzcd}
		\caption{Tukey ordering of $(\mathcal T_{3},<_T)$}
	\end{figure}
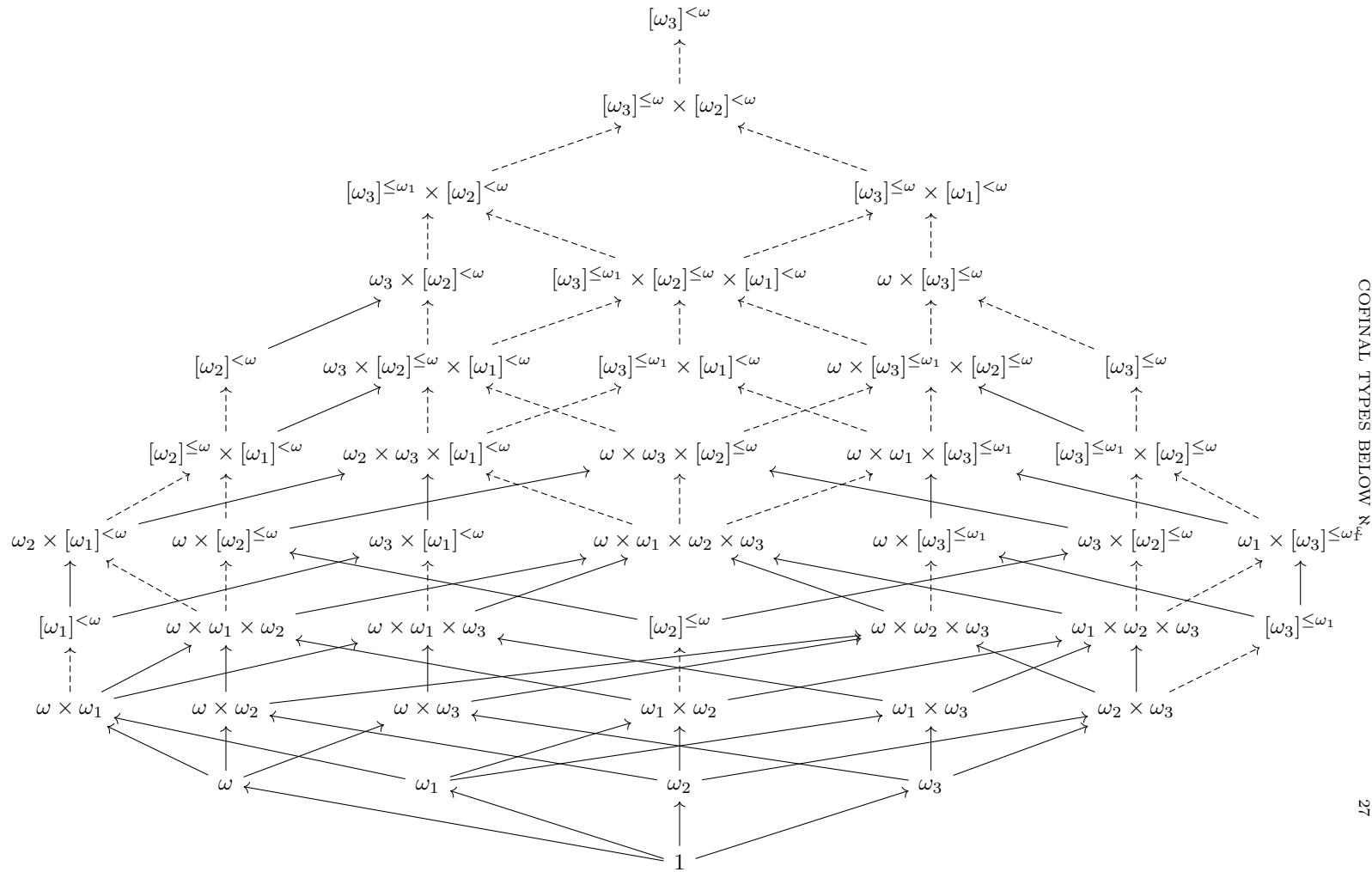
\end{landscape}


\begin{thebibliography}{Tuk40}
	
	\bibitem[Bir37]{MR1503323}
	Garrett Birkhoff.
	\newblock Moore-{S}mith convergence in general topology.
	\newblock {\em Ann. of Math. (2)}, 38(1):39--56, 1937.
	
	\bibitem[Day44]{MR9970}
	Mahlon~M. Day.
	\newblock Oriented systems.
	\newblock {\em Duke Math. J.}, 11:201--229, 1944.
	
	\bibitem[IR22]{paper47}
	Tanmay Inamdar and Assaf Rinot.
	\newblock Was {U}lam right? {I}: Basic theory and subnormal ideals.
	\newblock {\em Topology Appl.}, 2022.
	\newblock Accepted December 2021.
	
	\bibitem[Isb65]{MR201316}
	J.~R. Isbell.
	\newblock The category of cofinal types. {II}.
	\newblock {\em Trans. Amer. Math. Soc.}, 116:394--416, 1965.
	
	\bibitem[Isb72]{MR294185}
	John~R. Isbell.
	\newblock Seven cofinal types.
	\newblock {\em J. London Math. Soc. (2)}, 4:651--654, 1972.
	
	\bibitem[KT21]{kuzeljevic2021cofinal}
	Borisa Kuzeljević and Stevo Todor\v{c}evi\'{c}.
	\newblock Cofinal types on $\omega_2$, 2021.
	
	\bibitem[Pou80]{Pouzet}
	M~Pouzet.
	\newblock Parties cofinales des ordres partiels ne contenant pas
	d’antichaines infinies.
	\newblock {\em preprint}, 1980.
	
	\bibitem[RT14]{MR3247063}
	Dilip Raghavan and Stevo Todor\v{c}evi\'{c}.
	\newblock Combinatorial dichotomies and cardinal invariants.
	\newblock {\em Math. Res. Lett.}, 21(2):379--401, 2014.
	
	\bibitem[Sch55]{MR76836}
	J\"{u}rgen Schmidt.
	\newblock Konfinalit\"{a}t.
	\newblock {\em Z. Math. Logik Grundlagen Math.}, 1:271--303, 1955.
	
	\bibitem[She10]{Sh_922}
	Saharon Shelah.
	\newblock Diamonds.
	\newblock {\em Proc. Amer. Math. Soc.}, 138(6):2151--2161, 2010.
	
	\bibitem[Sta15]{MR3467982}
	Richard~P. Stanley.
	\newblock {\em Catalan numbers}.
	\newblock Cambridge University Press, New York, 2015.
	
	\bibitem[Tod85]{MR792822}
	Stevo Todor\v{c}evi\'{c}.
	\newblock Directed sets and cofinal types.
	\newblock {\em Trans. Amer. Math. Soc.}, 290(2):711--723, 1985.
	
	\bibitem[Tod89]{MR980949}
	Stevo Todor\v{c}evi\'{c}.
	\newblock {\em Partition problems in topology}, volume~84 of {\em Contemporary
		Mathematics}.
	\newblock American Mathematical Society, Providence, RI, 1989.
	
	\bibitem[Tod96]{MR1407459}
	Stevo Todor\v{c}evi\'{c}.
	\newblock A classification of transitive relations on {$\omega_1$}.
	\newblock {\em Proc. London Math. Soc. (3)}, 73(3):501--533, 1996.
	
	\bibitem[Tuk40]{MR0002515}
	John~W. Tukey.
	\newblock {\em Convergence and {U}niformity in {T}opology}.
	\newblock Annals of Mathematics Studies, No. 2. Princeton University Press,
	Princeton, N. J., 1940.
	
\end{thebibliography}
\end{document}